\documentclass{amsart}

\usepackage[english]{babel}

\usepackage{enumerate}
\usepackage{wasysym}
\usepackage{stmaryrd}
\usepackage{calrsfs}
\usepackage{mathrsfs}
\usepackage{rotating}
\usepackage[mathcal]{euscript}
\usepackage{amsfonts}
\usepackage{amsthm}
\usepackage{amssymb}
\usepackage{hyperref}
\usepackage{mathtools}

\input xy
\xyoption{all}

\newtheorem{theorem}{Theorem}[section]
\newtheorem{lemma}[theorem]{Lemma}
\newtheorem{conj-prop}[theorem]{Conjectural proposition}
\newtheorem{proposition}[theorem]{Proposition}
\newtheorem{conjecture}[theorem]{Conjecture}
\newtheorem{corollary}[theorem]{Corollary}

\theoremstyle{definition}
\newtheorem{definition}[theorem]{Definition}
\newtheorem{example}[theorem]{Example}

\theoremstyle{remark}
\newtheorem{remark}[theorem]{Remark}

\numberwithin{equation}{section}

\begin{document}

\title[Structure of node polynomials for curves on surfaces]{Structure of node polynomials\\for curves on surfaces}

\author{Nikolay Qviller}
\address{Centre of Mathematics for Applications, University of Oslo, P.O. Box 1053 Blindern, NO-0316 Oslo, NORWAY}
\email{nikolayq@cma.uio.no}



\date{\today}

\begin{abstract}
We provide a structural generalization of a theorem by Kleiman--Piene, concerning the enumerative geometry of nodal, algebraic curves in a complete linear system $|\mathscr{L}|$ on a smooth projective surface $S$. Provided that $r,$ the number of nodes, is sufficiently small compared to the ampleness of the linear system, we show that the number of $r$-nodal curves passing through points in general position on $S$ is given by a Bell polynomial in universally defined integers $a_{i}(S,\mathscr{L}),$ which we identify, using classical intersection theory, as linear, integral polynomials evaluated in four basic Chern numbers. Furthermore, we provide a decomposition of the $a_i$ as a sum of three terms with distinct geometric interpretations, and discuss the relationship between these polynomials and Kazarian's Thom polynomials for multisingularities of maps.
\end{abstract}

\maketitle

\setcounter{tocdepth}{1}

\tableofcontents

\section{Introduction}\label{sec:intro}
\subsection{Background}\label{subsec:background}
The enumerative geometry of nodal curves has, in recent years, grown into a rich and increasingly intriguing field of mathematics. While many of the questions which arise in this context belong naturally to the domain of classical algebraic geometry, there are also deep connections to more sophisticated, modern notions, such as mirror symmetry. In this paper, we consider the enumeration of nodal curves on surfaces, which we assume to be complex, projective (for natural reasons) and smooth and irreducible (for convenience). There have recently been important breakthroughs in this field. In particular, in 2010 Tzeng gave a first proof \cite{Tzeng} of important conjectures of G\" ottsche.

More precisely, let $S$ denote a surface as specified above. If $\mathscr{L}$ is a line bundle on $S,$ one may consider the associated complete linear system of curves, given by $|\mathscr{L}|,$ that is, $\mathbb{P}(H^{0}(S,\mathscr{L})).$ Denote this projective space by $Y,$ let $N$ be the dimension of $Y,$ and let $r \leq N$ be a non-negative integer. Denote by $N_{r}(S,\mathscr{L})$ the degree of the locus of $r$-nodal curves in $Y.$ Finally, let $(\partial,k,s,x)$ denote the four \textit{Chern numbers} of the polarized surface $(S,\mathscr{L}),$ that is, $\partial := \mathscr{L}^{2}, k = \mathscr{LK}_{S}, s = \mathscr{K}_{S}^{2}, x = c_{2}(S),$ where $\mathscr{K}_{S}$ denotes the canonical bundle on $S,$ and, for two line bundles $\mathscr{L}$ and $\mathscr{K},$ we let $\mathscr{LK} \in \mathbb{Z}$ denote the degree of $c_{1}(\mathscr{L})c_{1}(\mathscr{K}).$ The two primary conjectures of G\"ottsche (proved by Tzeng) are:

\begin{conjecture}\label{conj:polynomiality}
\emph{(\cite{Got}, Conjecture 2.1.)}
There exist polynomials $Z_{r} \in \mathbb{Q}[t,u,v,w]$ of degree $r$ (for $r \geq 0$) such that whenever $\mathscr{L}$ is $(5r-1)$-very ample, $N_{r}(S,\mathscr{L})$ is given by $Z_{r}(\partial,k,s,x).$
\end{conjecture}

\begin{conjecture}\label{conj:gen_got}
\emph{(\cite{Got}, Conjecture 2.4.)}
Let $(S,\mathscr{L})$ be fixed, then the generating function of the (virtual) curve numbers $Z_{r}(\partial,k,s,x)$ is
\begin{displaymath}
\sum_{r \geq 0} Z_{r}(\partial,k,s,x)(DG_{2}(\tau))^{r} = \frac{(DG_{2}(\tau)/q)^{\chi(\mathscr{L})}B_{1}(q)^{\mathscr{K}_{S}^{2}}B_{2}(q)^{\mathscr{L} \mathscr{K}_{S}}}{(\Delta(\tau)D^{2}G_{2}(\tau)/q^{2})^{\chi(\mathscr{O}_{S})/2}}.
\end{displaymath}
\end{conjecture}
Here, $G_{2}(\tau)$ is the second Eisenstein series and $\Delta(\tau)$ is the Ramanujan discriminant modular form. Let $q := e^{2\pi i\tau},$ then
\begin{eqnarray*}
G_{2}(\tau) & = & -1/24 + \sum_{n=1}^{\infty}\left( \sum_{d|n}d \right) q^{n}, \\
\Delta(\tau) & =  & q \prod_{m > 0} (1-q^{m})^{24}. 
\end{eqnarray*}
$D$ denotes the differential operator $q\frac{d}{dq},$ and finally $B_{1}(q)$ and $B_{2}(q)$ are (currently unknown) rational power series in $q.$

The latter result will be referred to as the \textit{G\"ottsche--Yau--Zaslow formula.} It involves five universal power series, three of which are quasi-modular forms, while the remaining two, $B_{1}(q)$ and $B_{2}(q),$ are not yet identified. However, using the recursive formula of Caporaso--Harris \cite{CH}, G\"ottsche computed the terms of these power series up to degree 28 \cite[Remark 2.5]{Got}.

In \cite{KST}, Kool, Shende and Thomas published a shorter proof of the first conjecture mentioned above. They also refined the result, showing that it is sufficient for $\mathscr{L}$ to be $r$-very ample. On the other hand, in \cite[Theorem 2.1]{Qvi}, we show that a consequence of the G\"ottsche--Yau--Zaslow formula is that the node polynomials $Z_{r}(\partial,k,s,x)$ (using terminology introduced by Kleiman and Piene) are of a very particular form:

\begin{theorem}\label{thm:main}
\emph{(\cite{Qvi}, Theorem 2.1.)}
For all $i \geq 1$ there exists a linear form $a_i$ in four variables, with coefficients which are integers, such that for all $r \geq 0,$
\begin{displaymath}
Z_{r}(\partial,k,s,x) = \frac{P_{r}(a_{1}(\partial,k,s,x), \ldots, a_{r}(\partial,k,s,x))}{r!},
\end{displaymath}
with $P_{r}$ the $r$th complete exponential Bell polynomial.
\end{theorem}
This theorem generalizes the structural part of a theorem by Kleiman--Piene, \cite[Theorem 1.1]{KP1}, concerning node polynomials for $r \leq 8$ nodes. It does not, however, give the numerical expressions of the polynomials $a_i,$ of which Kleiman--Piene computed the first eight:
\footnotesize
\begin{eqnarray*}
a_{1} & = & 3\partial + 2k + x \\
a_{2} & = & -42\partial - 39k - 6s - 7x \\
a_{3} & = & 1380\partial + 1576k + 376s + 138x \\
a_{4} & = & -72360\partial - 95670k - 28842s - 3888x \\
a_{5} & = & 5225472 \partial + 7725168 k + 2723400 s + 84384 x \\
a_{6} & = & -481239360 \partial -778065120 k - 308078520 s + 7918560 x\\
a_{7} & = & 53917151040\partial + 93895251840k + 40747613760s - 2465471520x\\
a_{8} & = & -7118400139200\partial - 13206119880240k - 6179605765200s + 516524964480x.
\end{eqnarray*}
\normalsize

The aim of this paper is to provide an explicit construction of the linear polynomials $a_i,$ with methods from intersection theory. As the direct computation of the node polynomials $Z_{r}$ becomes increasingly difficult for high values of $r,$ our emphasis is on the structure of these polynomials, which do indeed seem to have some striking combinatorial properties. Using the principle of inclusion-exclusion combined with excess intersection theory, multiple-point formulas, and finally residual intersection theory, we are able provide a natural decomposition of the polynomials $a_i$ into a sum of three terms with distinct geometric interpretations. Two of these terms are computable with the methods at hand. In addition, we point out the connections between the polynomials $a_i$ and the multisingularity (Thom) polynomials appearing in \cite{Kaz} by Kazarian.

\subsection{Structure of this article}\label{subsec:structure} In Section \ref{sec:basic_setup} we describe the schemes which will be used to construct the node polynomials from an intersection theoretical viewpoint. Section \ref{sec:shape} provides an ad hoc definition of integers $a_{i}(S,\mathscr{L}),$ depending on $S$ and $\mathscr{L},$ and associated classes $a_{i}(S,\mathscr{L})H^{i}$ in the Chow ring of the linear system of curves ($H$ being the class of a hyperplane). It then presents the node polynomials $Z_r$ as Bell polynomials evaluated in the integers $(-1)^{i-1}(i-1)!a_i(S,\mathscr{L}), 1 \leq i \leq r.$ Sections \ref{sec:equivalences} and \ref{sec:residual} discuss the various contributions to the integers $a_i(S,\mathscr{L})$ coming from different distinguished varieties of the intersection product that we study, and establish them as being the evaluation in the Chern numbers of $(S,\mathscr{L})$ of universally defined linear forms with integer coefficients. To avoid excessive notations, these forms are denoted by $a_i.$

\subsection{Conventions}\label{subsec:conventions} For a class $\alpha \in A^{k}(\mathbb{P}^{N}),$ we denote by $\int \alpha$ the degree of the class $\alpha \cdot H^{N-k} \in A^{N}(\mathbb{P}^{N}),$ with $H$ the class of a hyperplane. If $Y$ is a $\mathbb{C}$-scheme and $F$ is a scheme over $Y,$ we denote by $F^{\times r}$ the $r$-fold fiber product of $F$ with itself over $Y.$

\subsection{Acknowledgements}\label{subsec:acknow} I am greatly indebted to my advisor, Ragni Piene, who presented the initial idea to me and has steadily guided me towards the present article, answering all my questions with never-failing patience. An important part of the research which led to this paper was done while the author was a visiting student at MIT in the spring of 2012. It is a great pleasure to thank the Department of Mathematics and Steven Kleiman for hosting me. I would also like to thank Paolo Aluffi for an interesting and worthwile discussion.

\section{Intersection theoretical setup}\label{sec:basic_setup}
Let $S$ denote a smooth, irreducible projective surface over $\mathbb{C},$ and let $\mathscr{L}$ be a line bundle on $S;$ its global sections correspond to curves on $S,$ so we have a natural parameter space for curves, namely the projective space
\begin{equation}
Y := \mathbb{P}(H^{0}(S,\mathscr{L})).
\end{equation}
Let $N := \textnormal{dim }Y$ and set $F := S \times Y$ with projection $\gamma_{1}$ to $Y.$ Consider the relative effective divisor $\mathscr{D}$ in $F$ which is the total space of the complete linear system $|\mathscr{L}|;$ set-theoretically, it consists of pairs $(\kappa,y)$ such that $\kappa$ is a point on the curve $D_{y} \subset S$ corresponding to $y \in Y.$ Let $X \subset \mathscr{D}$ be the \textit{critical locus}, i.e., the scheme-theoretic closure of the set of pairs $(\kappa,y) \in S \times Y$ such that $\kappa$ is a singularity on $D_{y}.$ We consider $X$ as a scheme over $Y$ through the composition $f: X \stackrel{\iota}{\hookrightarrow} S \times Y \stackrel{\gamma_{1}}{\rightarrow} Y.$ Let $\widetilde{\mathscr{L}}$ denote $\mathscr{L} \boxtimes \mathscr{O}_{Y}(1),$ an invertible sheaf on $F.$ Recall that the associated sheaf of first order principal parts is defined as
\begin{equation}
\mathscr{P}^{1}_{F/Y}(\widetilde{\mathscr{L}}) := p_{2\ast}\Bigl(p_{1}^{\ast}\widetilde{\mathscr{L}}/(\mathscr{I}^{2} \cdot p_{1}^{\ast}\widetilde{\mathscr{L}})\Bigr),
\end{equation}
where $p_{j}: F \times_{Y} F \rightarrow F$ are the projections and $\mathscr{I}$ is the ideal sheaf of the diagonal $\Delta_{F}$ in $F \times_{Y} F.$ This sheaf fits into the vertical exact sequence below:
\[
\xymatrix
{
& 0 \ar[d] \\
& \Omega^{1}_{F/Y} \otimes \widetilde{\mathscr{L}} \ar[d] \\
\mathscr{O}_{F} \ar[dr]^{z} \ar[r]^{z'} & \mathscr{P}^{1}_{F/Y}(\widetilde{\mathscr{L}}) \ar[d] \\
& \widetilde{\mathscr{L}} \ar[d] \\
& 0 \\
}
\]
Scheme-theoretically, $\mathscr{D}$ is defined as the zero scheme of a section $z$ of the invertible sheaf $\widetilde{\mathscr{L}},$ since $\mathscr{O}_{F}(\mathscr{D}) = \mathscr{L} \boxtimes \mathscr{O}_{Y}(1).$ The section $z$ induces a section $z'$ of $\mathscr{P}^{1}_{F/Y}(\widetilde{\mathscr{L}}).$ Scheme-theoretically, $X$ is the zero scheme of $z'.$ The vertical exact sequence above shows that $\mathscr{P}^{1}_{F/Y}(\widetilde{\mathscr{L}})$ is locally free of rank 3, so every component of $X$ has codimension at most 3 in $F$. In case of equality for all components, the class of $X,$ which we denote by $\xi := [X] \in A^{\ast}(F),$ is given by $c_{3}(\mathscr{P}^{1}_{F/Y}(\widetilde{\mathscr{L}})).$ 

\begin{proposition}\label{prop:isom}
There is an isomorphism of $\mathscr{O}_{X}$-modules between the $Y$-relative normal bundle of $X$ in $F,$ i.e., $N_{X}F/Y,$ and (the restriction to $X$ of) the sheaf $\mathscr{P}^{1}_{F/Y}(\widetilde{\mathscr{L}}).$
\end{proposition}

\begin{proof}
Let $\mathscr{I}$ denote the ideal of $X$ in $F,$ then $\mathscr{I}_{|X} \cong \mathscr{I}/\mathscr{I}^{2} \cong (N_{X}F/Y)^{\vee}.$ On the other hand, $X$ is defined by the section $z': \mathscr{O}_{F} \longrightarrow \mathscr{P}^{1}_{F/Y}(\widetilde{\mathscr{L}}).$ Taking the duals, we have a morphism
\begin{displaymath}
\mathscr{P}^{1}_{F/Y}(\widetilde{\mathscr{L}})^{\vee} \longrightarrow \mathscr{O}_{F}^{\vee} \cong \mathscr{O}_{F}
\end{displaymath}
whose image is the ideal sheaf $\mathscr{I}.$ Restricting to $X,$ we get a surjection $$\mathscr{P}^{1}_{F/Y}(\widetilde{\mathscr{L}})_{|X}^{\vee} \longrightarrow \mathscr{I}/\mathscr{I}^{2},$$ which is, in fact, an isomorphism since the sheaves have the same rank. The result follows.
\end{proof}

\begin{example}\label{ex:critical_locus_P2}
Consider $S = \mathbb{P}^{2}$ and the family of curves of degree $d,$ i.e., sections of $\mathscr{O}(d).$ Thus $Y = \mathbb{P}^{d(d+3)/2}.$ Let $\varphi \in \mathbb{C}[x_{0},x_{1},x_{2}, c_{ijk} | i + j + k = d]$ be the homogeneous polynomial of degree $d$ in $x_{0},x_{1}$ and $x_{2},$ and of degree 1 in the $c_{ijk}:$

\begin{displaymath}
\varphi := \sum_{i + j + k = d} c_{ijk}x_{0}^{i}x_{1}^{j}x_{2}^{k}.
\end{displaymath}
Then $\mathscr{D} = Z(\varphi)$ is a hypersurface in $S \times Y,$ whereas $X,$ which is the locus of singular curves with a marked singularity, appears, by the Jacobi criterion, as the complete intersection of the three hypersurfaces in $F$ determined by the vanishing of the three partial derivatives $\frac{\partial \varphi}{\partial x_{0}}, \frac{\partial \varphi}{\partial x_{1}}$ and $\frac{\partial \varphi}{\partial x_{2}}.$ As observed in \cite[\S 1.1]{Alu2}, it follows that $X$ is a $\mathbb{P}^{N-3}$-bundle over $\mathbb{P}^{2};$ in particular, it is smooth. \hfill $\blacksquare$
\end{example}

Above, we defined $\xi = [X] \in A^{\ast}(F).$ Pushing this class down to $Y$ by $\gamma_{1}$ yields an enumerative cycle class, in the following sense: $Y$ being projective of dimension $N$, its Chow ring is simply $A^{\ast}(Y) = \mathbb{Z}[H]/H^{N+1},$ with $H$ the class of a hyperplane. Therefore, $\gamma_{1\ast}\xi = a_{1}(S,\mathscr{L})H$ for an integer $a_{1}(S,\mathscr{L}),$ since dimension is preserved by pushdowns. The integer $a_{1}(S,\mathscr{L})$ is precisely the number $N_{1}(S,\mathscr{L})$ of 1-nodal curves in the linear system $|\mathscr{L}|$ through $N-1$ points in general position on $S.$

\begin{proposition}\label{prop:a1_linear}
The integer $a_{1}(S,\mathscr{L})$ is given by evaluating a linear polynomial in four variables in the four Chern numbers $(\partial,k,s,x)$ of $(S,\mathscr{L}).$ More precisely, we have
\begin{equation}
a_{1}(S,\mathscr{L}) = 3\partial + 2k + x.
\end{equation}
\end{proposition}

\begin{proof}
We have $a_{1}(S,\mathscr{L}) = \gamma_{1\ast}\xi,$ with $\xi \in A^{\ast}(F)$ the class of $X,$ i.e., $c_{3}(\mathscr{P}_{F/Y}^{1}(\widetilde{\mathscr{L}})).$ Hence, putting $v:=c_{1}(\widetilde{\mathscr{L}})$ and $w_{j} = c_{j}(\Omega^{1}_{F/Y})$ for $j=1,2,$ the exact sequence
\begin{displaymath}
0 \rightarrow \Omega^{1}_{F/Y} \otimes \widetilde{\mathscr{L}} \rightarrow \mathscr{P}^{1}_{F/Y}(\widetilde{\mathscr{L}}) \rightarrow \widetilde{\mathscr{L}} \rightarrow 0
\end{displaymath}
yields $\xi = v^{3} + v^{2}w_{1} + vw_{2},$ which is a class of codimension 3 on $F.$ Let $\nu$ and $\gamma_{1}$ be the projections from $F = S \times Y$ to $S$ and $Y,$ respectively. Let $L := c_{1}(\mathscr{L}),$ $K: = c_{1}(\mathscr{K}_{S})$ and $H$ be the class of a hyperplane in $Y.$ For simplicity, let $L, K$ and $H$ also denote their own pullbacks (via $\nu$ and $\gamma_{1}$) to $F.$ Then $v = L + H$ and $w_{j} = c_{j}(\Omega^{1}_{F/Y}) = c_{j}(\nu^{\ast}\Omega^{1}_{S}) = \nu^{\ast}c_{j}(\Omega^{1}_{S}).$ We therefore get $w_{1} = \nu^{\ast}c_{1}(\Omega^{1}_{S}) = \nu^{\ast}c_{1}(\textnormal{det }\Omega^{1}_{S}) = \nu^{\ast}K,$ whereas $w_{2} = \nu^{\ast}c_{2}(\Omega^{1}_{S}) = \nu^{\ast}c_{2}(S).$ This gives us
\begin{equation}
\xi = (L+H)^{3} + K(L+H)^{2} + x(L+H).
\end{equation}
This can be seen as a polynomial in $H,$ and when pushing down to $Y,$ only the terms of first order in $H$ survive, so $a_{1}(S,\mathscr{L})H = \gamma_{1\ast}\xi = (3L^{2})H + (2LK)H + xH.$ Hence we conclude that $a_{1}(S,\mathscr{L}) = 3\partial + 2k + x.$
\end{proof}

A natural candidate for a scheme parametrizing curves with $r$ marked nodes would be the fibered product $X \times_{Y} \ldots \times_{Y} X$ with $r$ factors (geometrically, the fiber product ensures that we get $r$ marked nodes on the same curve, represented by a point in $Y$). There are, however, two major problems, both of which appear already for $r = 2.$ Several loci appear in the scheme $X \times_{Y} X$:
\begin{enumerate}
 \item a locus parametrizing binodal curves with marked nodes;
\item the diagonal $\Delta_{X},$ parametrizing nodal curves with a marked node;
\item the cuspidal locus, parametrizing cuspidal curves with a marked cusp.
\end{enumerate}
The diagonal is an excess locus; its dimension is $N-1,$ while the expected dimension of $X \times_{Y} X$ is $N-2.$ The cuspidal locus has the correct dimension, and is embedded in the diagonal (since there is only one singularity). Consequently, if we remove the intersection theoretical contribution of $\Delta_{X}$ to the intersection product $X_1 \cdot X_2,$ we get (up to a multiplicative factor of 2, due to the intrinsic symmetry of $X \times_{Y} X$) the number of 2-nodal curves plus the number of cuspidal curves in $|\mathscr{L}|$. Subtracting this last number and dividing by 2 yields the number of binodal curves in $|\mathscr{L|}.$ 

Intersection theoretically, the procedure is to intersect the pullbacks $p_{i}^{\ast}\xi, i = 1,2,$ with $p_{i}$ the projections $F \times_{Y} F \rightarrow F,$ then remove a certain excess class $B_{2}$ which represents the proper contribution of the diagonal and the contribution of the embedded cuspidal locus to the intersection product. We then wish to find the pushdown to $Y$ of this rational equivalence class, i.e., the class
\begin{displaymath}
\gamma_{2\ast} \bigl((p_{1}^{\ast}\xi \cdot p_{2}^{\ast}\xi) - B_{2}\bigr) \in A^{2}(Y),
\end{displaymath}
where $\gamma_{2}: F \times_{Y} F \rightarrow Y$ is the natural projection.

It should be obvious that for higher values of $r,$ the problem of the diagonals becomes more and more intricate.

\begin{definition}\label{def:curly_bracket}
For $F$ a smooth scheme of dimension $n,$ and $\alpha \in A^{\ast}(F),$ we let $\{\alpha\}^{k}$ denote the $k$-codimensional part of $\alpha,$ an element in $A^{k}(F).$ Similarly, we let $\{\alpha\}_{k}$ denote the $k$-dimensional part, an element in $A_{k}(F).$ \hfill $\blacksquare$
\end{definition}

\begin{example}\label{ex:two_nodes}
We will illustrate in more detail the enumeration of 2-nodal curves in the above setting. The idea is to consider the intersection class $p_{1}^{\ast} \xi \cdot p_{2}^{\ast}\xi,$ and subtract the excess coming from the diagonal and the embedded cuspidal locus, supported on the diagonal. Cuspidal curves in $|\mathscr{L}|$ are enumerated by a polynomial which is provided in, for example, Kazarian's paper \cite[Example 10.2]{Kaz}. In his notation, this is $S_{A_{2}} = 12\partial + 12k + 2s +2x.$ The diagonal $\Delta_{X}$ being a set-theoretically connected component of the intersection $p_{1}^{-1}(X) \cap p_{2}^{-1}(X) \cong X^{\times 2},$ we can use Proposition 9.1.1 in \cite{Ful} to compute its proper contibution to the intersection product. In our case the computation takes place on $F^{\times 2},$ and we get a class in $A_{m}(F^{\times 2})$ where $m = \textnormal{dim}(F^{\times 2}) - \sum_{i=1}^{2} \textnormal{codim}(p_{i}^{-1}X, F^{\times 2}) = 4 + \textnormal{dim }Y - 2 \cdot 3 = \textnormal{dim }Y - 2,$ namely
\begin{equation}
\left\{c\Big(\left(p_{1}^{\ast}N_{X}F\right) |\Delta_{X}\Big) \cdot c\Big(\left(p_{2}^{\ast}N_{X}F\right) |\Delta_{X}\Big) \cdot c\left( N_{\Delta_{X}}F^{\times 2} \right)^{-1} \cap [\Delta_{X}] \right\}_{N-2},
\end{equation}
representing the contribution of the diagonal itself to $p_{1}^{\ast}\xi \cdot p_{2}^{\ast}\xi.$ We want to find the pushdown of this class to $Y$ through $\gamma_{2} = \gamma_{1} \circ p_{1}.$ Since $\Delta_{X} \hookrightarrow \Delta_{F} \hookrightarrow F^{\times 2}$ are two regular embeddings, the normal bundle of the first being $N_{X}F \cong \mathscr{P}^{1}_{F/Y}(\widetilde{\mathscr{L}})$ and the one of the second being the pullback of $T_{F/Y} \cong T_{S},$ the class introduced above is equal to
\begin{equation}
\left\{ c\left(\mathscr{P}^{1}_{F/Y}(\widetilde{\mathscr{L}})\right) \cdot c(T_{F/Y})^{-1} \cap [X] \right\}_{N-2}.
\end{equation}
Recall our notations $L := c_{1}(\mathscr{L}),$ $K := c_{1}(\mathscr{K}_{S}),$ and $H$ is the class of a hyperplane in $Y.$ We also use $v = c_{1}(\widetilde{\mathscr{L}}) = L+H$ and $w_{j} = c_{j}(\Omega^{1}_{F/Y}),$ so that $w_{1} = \nu^{\ast}K$ and $w_{2} = \nu^{\ast}x,$ where $\nu$ is the projection from $F$ to $S.$ Now, we have $c_{1}(T_{F/Y}) = -w_{1}$ and $c_{2}(T_{F/Y})=w_{2}$ since $T_{F/Y}^{\vee} \cong \Omega^{1}_{F/Y}.$ Thus, $c(T_{F/Y})^{-1} = 1+w_{1}+(w_{1}^{2}-w_{2}).$ On the other hand, the exact sequence
\begin{equation}
0 \rightarrow \Omega^{1}_{F/Y} \otimes \widetilde{\mathscr{L}} \rightarrow \mathscr{P}^{1}_{F/Y}(\widetilde{\mathscr{L}}) \rightarrow \widetilde{\mathscr{L}} \rightarrow 0
\end{equation}
yields, by the Whitney sum formula, $c(\mathscr{P}^{1}_{F/Y}(\widetilde{\mathscr{L}})) = c(\Omega^{1}_{F/Y} \otimes \widetilde{\mathscr{L}}) \cdot c(\widetilde{\mathscr{L}}).$ Thus, considering Chern polynomials:
\begin{eqnarray*}
c_{t}(\Omega^{1}_{F/Y} \otimes \widetilde{\mathscr{L}}) & = & \sum_{i=0}^{2} t^{i} (1 + tc_{1}(\widetilde{\mathscr{L}}))^{2-i}c_{i}(\Omega^{1}_{F/Y}) \\
& = & \bigl(1 + t(L+H)\bigr)^{2} + t\bigl(1+t(L+H)\bigr)w_{1} + t^{2}w_{2}.
\end{eqnarray*}
Also, we have $[X] = \xi = (L+H)^{3} + K(L+H)^{2} + x(L+H).$ What we want is the degree 2 part of the coefficient of $H^{2}$ in the expansion of
\begin{equation}
c(\Omega^{1}_{F/Y} \otimes \widetilde{\mathscr{L}}) \cdot c(\widetilde{\mathscr{L}}) \cdot c(T_{F/Y})^{-1} \cap [X],
\end{equation}
when considering $K$ and $L$ to have degree 1 and $x$ to have degree 2. A simple computation in, for instance, \verb+Maple+, yields the following polynomial:
\begin{equation}
Q_{2} := 18\partial + 15k + 2s + 3x.
\end{equation}
We see that $Q_{2} + 2S_{A_{2}} = 18\partial + 15k + 2s + 3x + 2 \cdot (12\partial + 12k + 2s +2x) = 42\partial+39k+6s+7x,$ which is precisely the polynomial $-a_{2}(\partial,k,s,x)$ of Kleiman--Piene. On the other hand, the pushdown to $Y$ of the intersection product $p_{1}^{\ast}\xi \cdot p_{2}^{\ast}\xi$ is equal to $a_{1}^{2}H^{2}$ where $a_{1}H = \gamma_{1\ast}\xi = (3\partial+2k+x)H.$ In total, the pushdown of the class representing honest 2-nodal curves is $(a_{1}^{2} + a_{2})H^{2}.$ Divide this by 2 to avoid recountings due to permutations of the nodes; the result is, up to a factor $H^{2},$ the number of 2-nodal curves through $N-2$ points in general position on $S.$ \hfill $\blacksquare$
\end{example}

\section{Shape of node polynomials}\label{sec:shape}
For greater values of $r$ there are several diagonals which appear, as well as their intersections, which we refer to as \textit{polydiagonals}. There is a bijection between polydiagonals in $X^{\times r}$ and non-singleton partitions $\pi$ of $[r] := \{1,\ldots, r\}.$ Indeed, a partition is of a set of disjoint subsets of $[r]$ whose union is equal to $[r].$ These subsets are called \textit{blocks} of the partition. Denote by $\Pi_{r}$ the set of all partitions of $[r],$ and by $\Pi_{r}^{\circ}$ the set of non-singleton partitions, the singleton partition being $\widehat{0}_{r} := 1|2|\ldots |r,$ i.e., the only partition with $r$ blocks. Then $\pi \in \Pi_{r}^{\circ}$ corresponds to the polydiagonal
\begin{equation}
\Delta^{(r)}_{\pi} := \{(x_{1}, \ldots, x_{r}) \in X^{\times r}, x_{i} = x_{j} \textnormal{ if } i \textnormal{ and } j \textnormal{ are in the same block of } \pi \}
\end{equation}
in $X^{\times r}.$ We denote by $\widehat{1}_{r}$ the single-block partition $12 \ldots r.$ If there is no room for confusion, we use $\widehat{0}$ and $\widehat{1}$ instead of $\widehat{0}_{r}$ and $\widehat{1}_{r}.$

It is a well-known fact that imposing $r$ nodes on the curves in a system is a codimension $r$ requirement. Hence the dimension of the \textit{configuration space} $\mathbb{F}(X,r)$ (i.e., the complement of the diagonals in $X^{\times r}$) is equal to $N - r,$ where $N = \textnormal{dim }Y.$ The union of the scheme-theoretic polydiagonals, however, is a connected component of $X^{\times r}$ of dimension $N-1,$ since it contains the small diagonal $\Delta^{(r)}_{12\ldots r} \cong X.$

Letting $p_{j}: F^{\times r} \rightarrow F, 1 \leq j \leq r,$ denote the projections, we make the following \textit{ad hoc} definition, whose importance will be made clear in the following:

\begin{definition}\label{def:equiv}
Let $r \geq 1.$ For each $\hat{0} \neq \pi \in \Pi_{r},$ we let $B^{(r)}_{\pi} \in A_{\ast}(\Delta^{(r)}_{\pi})$ denote the equivalence (in the sense of \cite[Definition 6.1.2]{Ful}) of the closed subset $\Delta^{(r)}_{\pi}$ for the intersection product $p_{1}^{\ast}\xi \cdot \ldots \cdot p_{r}^{\ast}\xi.$ Also, we let $B^{(r)}_{\hat{0}} \in A_{\ast}(X^{\times r})$ denote the intersection product itself. Furthermore, define
\begin{displaymath}
a_{i}(S,\mathscr{L}) := (-1)^{i-1}(i-1)!\int_{Y} f_{\ast}B^{(i)}_{1\ldots i} \in \mathbb{Z},
\end{displaymath}
where $f: X \rightarrow Y$ is the composition of the embedding $\iota: X \hookrightarrow F$ and the projection $F = S \times Y \rightarrow Y.$ \hfill $\blacksquare$
\end{definition}

\begin{remark}\label{rem:all_dist}
We would like to emphasize the fact that we are not simply considering the proper contribution of $\Delta^{(r)}_{\pi}$ to the intersection product $p_{1}^{\ast}\xi \cdot \ldots \cdot p_{r}^{\ast}\xi,$ but the contribution of all distinguished varieties whose support is contained in this polydiagonal. 
\end{remark}

\begin{definition}\label{def:comp_bell}
The \textit{complete (exponential) Bell polynomials} are defined by the formal identity in $t,$
\begin{equation}\label{eqn:formal_id}
\sum_{r \geq 0}P_{r}t^{r}/r! = \exp \left(\sum_{l \geq 1} x_{l}t^{l}/l! \right).
\end{equation}
\hfill $\blacksquare$
\end{definition}

\begin{example}\label{ex:bell_polys}
The first four Bell polynomials are easily seen to be:
\begin{eqnarray*}
P_1(x_1) & = &  x_1 \\
P_2(x_1,x_2) & = & x_1^2 + x_2 \\
P_3(x_1,x_2,x_3) & = & x_1^3 + 3x_1x_2 + x_3 \\
P_4(x_1,x_2,x_3,x_4) & = & x_1^4 + 6x_1^2x_2 + 4x_1x_3 + 3x_2^2 + x_4
\end{eqnarray*}
\hfill $\blacksquare$
\end{example}

One can also define partial Bell polynomials:

\begin{definition}\label{def:part_bell}
The \textit{partial Bell polynomials} are defined for all $n \geq 1$ and all $1 \leq l \leq n,$ by the following formula:
\footnotesize
\begin{displaymath}
P_{n,l}(x_{1}, x_{2}, \ldots, x_{n-l+1}) := \sum \frac{n!}{j_{1}!j_{2}! \ldots j_{n-l+1}!} \left(\frac{x_{1}}{1!}\right)^{j_{1}}\left(\frac{x_{2}}{2!}\right)^{j_{2}} \ldots \left(\frac{x_{n-l+1}}{(n-l+1)!}\right)^{j_{n-l+1}},
\end{displaymath}
\normalsize
where we sum over all tuples of integers $j_{1}, \ldots, j_{n-l+1} \geq 0$ such that $j_{1} + \ldots + j_{n-l+1} = l$ and $j_{1} + 2j_{2} + \ldots + (n-l+1)j_{n-l+1} = n.$ \hfill $\blacksquare$
\end{definition}

Combinatorically, the coefficient in front of $x_{1}^{j_{1}}x_{2}^{j_{2}} \ldots x_{n-l+1}^{j_{n-l+1}}$ is interpreted as the number of ways to partition a set of $n$ elements into $l$ blocks where $j_{1}$ blocks have 1 element, $j_{2}$ have 2 elements etc., the members of the set being indistinguishable. The complete Bell polynomials are the sum of the partial ones:
\begin{equation}
P_{n}(x_{1},\ldots,x_{n}) = \sum_{l=1}^{n}P_{n,l}(x_1,x_2,\ldots,x_{n-l+1}).
\end{equation}

The object of this section is to show the following theorem:

\begin{theorem}\label{thm:shape}
Let $(S,\mathscr{L})$ be a polarized smooth, irreducible projective surface over $\mathbb{C}$ and let $r \geq 1$ be an integer. Then, provided $\mathscr{L}$ is $r$-very ample, the number $N_{r}(S,\mathscr{L})$ of $r$-nodal curves in the linear system $|\mathscr{L}|$ is given by
\begin{displaymath}
N_{r}(S,\mathscr{L}) = \frac{P_{r}(a_{1}(S,\mathscr{L}), \ldots, a_{r}(S,\mathscr{L}))}{r!},
\end{displaymath}
where $P_{r}$ is the $r$th complete Bell polynomial.
\end{theorem}

Consider the fiber product $F^{\times r} = F \times_{Y} \ldots \times_{Y} F,$ with $r$ projections $p_{j}$ to $F.$ The $r$-fold fiber product $X \times_{Y} \ldots \times_{Y} X$ is equal to $p_{1}^{-1}(X) \cap \ldots \cap p_{r}^{-1}(X).$ As a starting point for enumerating $r$-nodal curves in $|\mathscr{L}|,$ one could consider the intersection product
\begin{displaymath}
p_{1}^{\ast}\xi \cdot \ldots \cdot p_{r}^{\ast}\xi \in A^{\ast}(F^{\times r}).
\end{displaymath}
However, the polydiagonals give an excess contribution to this intersection, which we want to remove. This motivates the following definition: 


\begin{definition}\label{def:class_I}
We denote by $I_{r}$ the intersection class $p_{1}^{\ast}\xi \cdot \ldots \cdot p_{r}^{\ast}\xi$ minus the equivalence of the union of the polydiagonals. More precisely, recall that $\Pi_{r}^{\circ}$ denotes the set of partitions of $[r],$ $1|2|\ldots|r$ excluded, then
\begin{equation}
I_{r} := p_{1}^{\ast}\xi \cdot \ldots \cdot p_{r}^{\ast}\xi - \left(p_{1}^{\ast}\xi \cdot \ldots \cdot p_{r}^{\ast}\xi\right)^{\bigcup_{\pi \in \Pi_{r}^{\circ}}\Delta^{(r)}_{\pi}}.
\end{equation} \hfill $\blacksquare$
\end{definition}

We now want to express $I_r$ using the classes  $B^{(r)}_{\pi}.$ For this, we need some notation. If $\pi$ and $\pi'$ are two partitions in $\Pi_{r},$ we write $\pi' \prec \pi$ if each block of $\pi'$ is contained in a block of $\pi,$ i.e., if the partition $\pi'$ is a refinement of the partition of $\pi.$ The number of blocks of a partition $\pi$ is denoted by $|\pi|.$ Thus, the singleton partition $\widehat{0} = 1|2|\ldots|r$ is the only partition $\pi$ of $[r]$ such that $|\pi| = r.$

\begin{lemma}\label{lemma:mobius_coeffs}
We have
\begin{equation}
I_{r} = \sum_{\pi \in \Pi_{r}} n^{(r)}_{\pi} B^{(r)}_ {\pi},
\end{equation}
where the coefficients $\{n^{(r)}_{\pi}\}$ are defined as follows: For $\pi \in \Pi_{r},$ let $s_{i}(\pi)$ denote the number of blocks of size $i$ in $\pi,$ where $1 \leq i \leq r.$ Then
\begin{equation}
n^{(r)}_{\pi} = \prod_{i=1}^{r} \left[(-1)^{i-1}(i-1)!\right]^{s_{i}(\pi)}.
\end{equation}
\end{lemma}

\begin{proof}
We have
\begin{eqnarray*}
I_{r} & = & p_{1}^{\ast}\xi \cdot \ldots \cdot p_{r}^{\ast}\xi - \left(p_{1}^{\ast}\xi \cdot \ldots \cdot p_{r}^{\ast}\xi\right)^{\bigcup_{\pi \in \Pi_{r}^{\circ}}\Delta^{(r)}_{\pi}} \\
& = & p_{1}^{\ast}\xi \cdot \ldots \cdot p_{r}^{\ast}\xi - \sum_{Z \subseteq \bigcup_{\pi \in \Pi_{r}^{\circ}}\Delta^{(r)}_{\pi}} \left(p_{1}^{\ast}\xi \cdot \ldots \cdot p_{r}^{\ast}\xi\right)^{Z},
\end{eqnarray*}
where the $Z$s appearing in the index are distinguished varieties of the intersection product $p_{1}^{\ast}\xi \cdot \ldots \cdot p_{r}^{\ast}\xi$ support on the union of the polydiagonals. Since these $Z$s are irreducible, we have
\begin{equation}
-\sum_{Z \subseteq \bigcup_{\pi \in \Pi_{r}^{\circ}}\Delta^{(r)}_{\pi}} \left(p_{1}^{\ast}\xi \cdot \ldots \cdot p_{r}^{\ast}\xi\right)^{Z} =  \sum_{\pi \in \Pi_{r}^{\circ}}\sum_{Z \subseteq \Delta^{(r)}_{\pi}} n^{(r)}_{\pi} \left(p_{1}^{\ast}\xi \cdot \ldots \cdot p_{r}^{\ast}\xi\right)^{Z},
\end{equation}
where the $n^{(r)}_{\pi}$ are defined so that each term $\left(p_{1}^{\ast}\xi \cdot \ldots \cdot p_{r}^{\ast}\xi\right)^{Z}$ for some distinguished variety $Z$ supported on the union of the diagonals occurs only once. Starting with the ``largest'' polydiagonals, i.e., the $\Delta^{(r)}_{\pi}$ for which $|\pi| = r-1,$ the coefficient $n^{(r)}_{\pi}$ must be $-1.$ Then each term $\left(p_{1}^{\ast}\xi \cdot \ldots \cdot p_{r}^{\ast}\xi\right)^{Z}$ for $Z$ supported on some polydiagonal $\Delta^{(r)}_{\pi}$ with $|\pi| = r-2$ occurs $\sum_{\pi' \prec \pi} n^{(r)}_{\pi'}$ times, hence we must add them to the previous expression, but with a coefficient
\begin{equation}
n^{(r)}_ {\pi} := -1 - \sum_{\pi' \prec \pi, |\pi'| \neq r} n^{(r)}_{\pi'} = - \sum_{\pi' \prec \pi} n^{(r)}_{\pi'}
\end{equation}
to ensure they are only subtracted once. Now continue this way, using the principle of inclusion-exclusion. We recognize the definition of the coefficients $n^{(r)}_{\pi}$ as $n^{(r)}_{\pi} = \mu(\widehat{0}_r,\pi)$ with $\mu$ the M\"obius function of the poset $\Pi_{r}$ (cf. \cite[Section 3.9]{Sta}). Since we have
\begin{equation}
\mu_{n} := \mu(\widehat{0}_{n},\widehat{1}_{n}) = (-1)^{n-1}(n-1)! 
\end{equation}
by \cite[Example 3.10.4]{Sta}, and because of the product theorem for M\"obius functions \cite[Proposition 3.8.2]{Sta}, it follows that
\begin{equation}
n^{(r)}_{\pi} = \prod_{i=1}^{r} \left[(-1)^{i-1}(i-1)!\right]^{s_{i}(\pi)}.
\end{equation}
\end{proof}


For each $r \geq 1,$ it is clear that polydiagonals in $X^{\times r}$ are isomorphic, as schemes, to fibered products of small diagonals from the $X^{\times i}, i \leq r.$ For instance, in $X^{\times 6}$ we have
\begin{equation}
\Delta^{(6)}_{1|23|456} \cong X \times_{Y} \Delta^{(2)}_{12} \times_{Y} \Delta^{(3)}_{123}.
\end{equation}
So when passing from fewer than $r$ to $r$ nodes, what is new compared to previous cases --- from a structural point-of-view --- is the contribution to the intersection product $p_{1}^{\ast}\xi \cdot \ldots \cdot p_{r}^{\ast}\xi$ from the small diagonal $\Delta^{(r)}_{12\ldots r}.$ From the above, this contribution appears with the coefficient $(-1)^{r-1}(r-1)!,$ which is what motivates Definition \ref{def:equiv}.

Since $B^{(i)}_ {1\ldots i}$ is a class of dimension $N-i,$ the codimension of its pushdown in $Y$ becomes $i.$ We want to show that $\forall r \geq 2,$ the class $I_{r} = \sum_{\pi \in \Pi_{r}} n^{(r)}_{\pi} B^{(r)}_{\pi} \in A_{N-r}(F^{\times r})$ (each term having been pushed forward to a class on $F^{\times r}$) pushes down to the $r$th Bell polynomial in the classes $a_{i}(S,\mathscr{L})H^{i}, 1 \leq i \leq r$ on $Y.$ We need an intermediate result (to lighten the notation, we assume all classes are pushed forward to the appropriate ambient variety $F^{\times i}$):

\begin{proposition}\label{prop:splitting}
For any $r \geq 2$ and any $\pi \in \Pi_{r},$ we have the following equality of classes on $Y$ ($\prod$ denoting the intersection product $\cdot$):
\begin{equation}
\gamma_{r\ast}B^{(r)}_{\pi} = \prod_{i=1}^{r} \Bigl(\gamma_{i\ast} B^{(i)}_{1\ldots i} \Bigr)^{s_{i}(\pi)} \in A^{r}(Y).
\end{equation}
\end{proposition}

Before proving the proposition, let us clarify by looking at a simple example.

\begin{example}\label{ex:five_diags}
Say $r = 5$ and we are interested in the contribution to the intersection product $p_{1}^{\ast}\xi \cdot \ldots \cdot p_{5}^{\ast}\xi \in A^{\ast}(F^{\times r})$ coming from the diagonal $\Delta_{12|345}.$ For notational simplicity, let $p$ and $q$ denote the projections $p_{12}$ and $p_{345}$ from $F^{5}$ to $F^{\times 2}$ and $F^{\times 3},$ respectively. Then there are two natural ways of associating a class on $Y$ to the class $B^{(5)}_{12|345}.$ The ``easiest'' is to push forward by $\gamma_{5}.$ The other one consists of pushing forward to $F^{\times 2} \times F^{\times 3}$ through $p \times q,$ then to $Y \times Y$ with $\gamma_{2} \times \gamma_{3},$ and finally pulling back to $Y$ via the the diagonal embedding $\delta_{Y}: Y \hookrightarrow Y \times Y.$ The diagram
\[
\xymatrix
{
F^{\times 5} \ar[r]^>>>>>{p \times q} \ar_{\gamma_{5}}[d] & F^{\times 2} \times F^{\times 3} \ar^{\gamma_{2} \times \gamma_{3}}[d]\\
Y \ar@{^{(}->}_{\delta_{Y}}[r] & Y \times Y
}
\]
is a fiber square, and by \cite[Proposition 1.7]{Ful}, the relation
\begin{equation}
\label{equality}
\gamma_{5\ast} (p \times q)^{\ast} \alpha = \delta_{Y}^{\ast}(\gamma_{2} \times \gamma_{3})_{\ast} \alpha \in A^{\ast}(Y)
\end{equation}
holds $\forall \alpha \in A^{\ast}(F^{\times 2} \times F^{\times 3}).$ There is a degree-preserving morphism of graded rings
\begin{displaymath}
A^{\ast}(F^{\times 2}) \otimes A^{\ast}(F^{\times 3}) \stackrel{\times}{\rightarrow} A^{\ast}(F^{\times 2} \times F^{\times 3}),
\end{displaymath}
called the \textit{exterior product,} and the relation (\ref{equality}) holds for all $\alpha$ in its image. However, the intersection product $\cdot$ on $Y$ is simply the composition
\begin{displaymath}
A^{\ast}(Y) \otimes A^{\ast}(Y) \stackrel{\times}{\rightarrow} A^{\ast}(Y \times Y) \stackrel{\delta_{Y}^{\ast}}{\rightarrow} A^{\ast}(Y).
\end{displaymath}
Let $\alpha$ be the exterior product of $B^{(2)}_{12}$ and $B^{(3)}_{123}.$ Then the right hand side of (\ref{equality}) is $\gamma_{2\ast}B^{(2)}_{12} \cdot \gamma_{3\ast}B^{(3)}_{123}.$ So to conclude that $\gamma_{5\ast}B^{(5)}_{12|345} = \gamma_{2\ast}B^{(2)}_{12} \cdot \gamma_{3\ast}B^{(3)}_{123},$ it suffices to have the equality $(p \times q)^{\ast}\alpha = B^{(5)}_{12|345}.$ But $(p \times q)^{\ast}\alpha = p^{\ast}B^{(2)}_{12} \cdot q^{\ast}B^{(3)}_{123},$ so we must show that this intersection product equals $B^{(5)}_{12|345}.$ \hfill $\blacksquare$
\end{example}

In fact, what is done in the preceding example is general:

\begin{lemma}\label{lemma:splitting}
Let $r \geq 2$ and consider a partition $\pi \in \Pi_{r}.$ For each block of $\pi$ there is a corresponding subset $I$ of $[r].$ Consider the natural projection $p_{I}: F^{\times r} \rightarrow F^{\times |I|}.$ Denote the set of blocks of $\pi$ by $\mathbb{B}(\pi).$ Then the pushdown to $Y$ through $\gamma_{r}$ of the class
\begin{displaymath}
\prod_{I \in \mathbb{B}(\pi)} p_{I}^{\ast}B^{(|I|)}_{1 \ldots |I|} \in A^{\ast}(F^{\times r})
\end{displaymath}
is equal to the intersection product over $I \in \mathbb{B}(\pi)$ of the classes $\gamma_{|I|\ast}B^{(|I|)}_{1 \ldots |I|} \in A^{\ast}(Y).$
\end{lemma}

\begin{proof}
The matter of generalizing the result from the previous example is purely formal, and therefore left out.
\end{proof}

We now prove Proposition \ref{prop:splitting}:

\begin{proof}
By Lemma \ref{lemma:splitting}, it suffices to show that after push-forward to $F^{\times r},$
\begin{equation}
B^{(r)}_{\pi} = \prod_{I \in \mathbb{B}(\pi)} p_{I}^{\ast}B^{(|I|)}_{1\ldots |I|}.
\end{equation}
For each $1 \leq i \leq r,$ let $p_{i}$ denote the $i$th projection from $F^{\times r}$ to $F$ and $\delta_{r}$ the diagonal embedding of $F^{\times r}$ in $F^{\times r} \times \ldots \times F^{\times r}.$ Let $N$ be the dimension of $Y.$ We are interested in the intersection diagram
\[
\xymatrix
{
\bigcap X_{i} \cong X^{\times r} \ar@{^{(}->}[r] \ar@{^{(}->}[d] & F^{\times r} \ar@{^{(}->}^{\delta_{r}}[d]\\
X_{1} \times \ldots \times X_{r} \ar@{^{(}->}[r] & F^{\times r} \times \ldots \times F^{\times r}
}
\]
where $X_{i} := p_{i}^{-1}(X).$ Denote by $\mathscr{N}^{(r)}$ the pullback of the normal bundle of $X_{1} \times \ldots \times X_{r}$ in $F^{\times r} \times \ldots \times F^{\times r}.$ The latter embedding is closed regular of codimension $3r,$ so $\mathscr{N}^{(r)}$ is a bundle of rank $3r$ on $X^{\times r}.$ Let $\zeta_{r}$ be the projection $\mathscr{N}^{(r)} \rightarrow X^{\times r}.$ The cone $C^{(r)} := C_{X^{\times r}}F^{\times r},$ which has pure dimension $2r + N,$ embeds as a closed subcone of $\mathscr{N}^{(r)}$ over $X^{\times r},$ and gives a cycle $[C^{(r)}]$ of dimension $2r + N$ on this bundle.

Let the irreducible components of $C^{(r)}$ be $C^{(r)}_{j}, 1 \leq j \leq t_{r},$ with geometric multiplicities $m^{(r)}_{j}$ and supports $Z^{(r)}_{j},$ which are irreducible subschemes of $X^{\times r},$ not necessarily all distinct. Let $z^{(r)}_{j}: Z^{(r)}_{j} \rightarrow N^{(r)}_{j}$ be the zero section of the restriction of $\mathscr{N}^{(r)}$ to $Z^{(r)}_{j}.$ 

Now, $B^{(r)}_{\pi}$ is defined as the sum of the contributions to $X_1 \cdot \ldots \cdot X_r$ coming from all distinguished varieties $Z^{(r)}_{j}$ (defined above) supported on $\Delta^{(r)}_{\pi} = \bigcap_{I \in \mathbb{B}(\pi)} \Delta^{(r)}_{I}.$ Hence, with $p_I: F^{\times r} \rightarrow F^{\times |I|}$ the natural projection for each $I \in \mathbb{B}(\pi),$ there is a multiplicative correspondence between tuples of components of $C^{(|I|)}, I \in \mathbb{B}(\pi),$ each with support contained in $\Delta^{(|I|)}_{1\ldots |I|},$ and components of $C^{(r)}$ with support contained in $\Delta^{(r)}_{\pi},$ such that the geometric multiplicity of $C^{(r)}_{j}$ equals the product of the geometric multiplicities of the corresponding components of the $C^{(|I|)}.$ Hence, letting $\delta$ denote the diagonal embedding of $F^{\times r}$ into $F^{\times r} \times \ldots \times F^{\times r}$ ($|I|$ factors) and letting $\times$ denote the exterior product,
\begin{eqnarray*}
\prod_{I \in \mathbb{B}(\pi)} p_{I}^{\ast}B^{(|I|)}_{1\ldots |I|} & = & \delta^{\ast}\left(\bigtimes_{I \in \mathbb{B}(\pi)} p_{I}^{\ast}B^{(|I|)}_{1\ldots |I|}\right) \\
& = & \delta^{\ast} \left( \bigtimes_{I \in \mathbb{B}(\pi)} p_{I}^{\ast} \left(\sum_{Z^{(|I|)}_{j} \subseteq \Delta^{(|I|)}_{1\ldots |I|}} m^{(|I|)}_{j}z^{(|I|)\ast}_{j}[C^{(|I|)}_{j}] \right)\right) \\
& = & \delta^{\ast} \left( \sum_{\substack{I \in \mathbb{B}(\pi) \\ Z^{(|I|)}_{j(I)} \subseteq \Delta^{(|I|)}_{1\ldots |I|}}} \prod_{I \in \mathbb{B}(\pi)} m^{(|I|)}_{j(I)} \bigtimes_{I \in \mathbb{B}(\pi)} p_{I}^{\ast} z^{(|I|)\ast}_{j(I)}[C^{(|I|)}_{j(I)}] \right) \\
& = & \sum_{\substack{I \in \mathbb{B}(\pi) \\ Z^{(|I|)}_{j(I)} \subseteq \Delta^{(|I|)}_{1\ldots |I|}}} \prod_{I \in \mathbb{B}(\pi)} m^{(|I|)}_{j(I)} \delta^{\ast} \left( \bigtimes_{I \in \mathbb{B}(\pi)} p_{I}^{\ast} z^{(|I|)\ast}_{j(I)}[C^{(|I|)}_{j(I)}] \right).
\end{eqnarray*}
Using the definition of the intersection product \cite[Section 6.1]{Ful} and the correspondence between the $C^{(r)}_{j}$ whose support is contained in $\Delta^{(r)}_{\pi},$ and tuples of components of the $C^{(|I|)}$ for $I \in \mathbb{B}(\pi),$ with $\prod_{I \in \mathbb{B}(\pi)} m^{(|I|)}_{j(I)} = m^{(r)}_{j},$ we now get
\begin{eqnarray*}
\prod_{I \in \mathbb{B}(\pi)} p_{I}^{\ast}B^{(|I|)}_{1\ldots |I|} & = & \sum_{Z^{(r)}_{j} \subseteq \Delta^{(r)}_{\pi}} m^{(r)}_{j} z^{(r)\ast}_{j}[C^{(r)}_{j}] \\
& = & B^{(r)}_{\pi},
\end{eqnarray*}
\end{proof}

We may now proceed to prove the main theorem of this section, Theorem \ref{thm:shape}, concerning the shape of the node polynomials:\\

\begin{proof} We assume $r$ is such that $\mathscr{L}$ is $r$-very ample. Hence, by Proposition 2.1 in \cite{KST}, a general $r$-dimensional linear system $\mathbb{P}^{r} \subset |\mathscr{L}|$ contains a finite number of $r$-nodal curves, appearing with multiplicity 1, and all other curves are reduced with geometric genus strictly larger than $g-r,$ where $2g-2 = \mathscr{L} \cdot (\mathscr{L}+\mathscr{K}_{S}).$ These curves are excluded from the counting by subtracting from $p_{1}^{\ast}\xi \cdot \ldots p_{r}^{\ast}\xi$ the equivalence of the polydiagonals. Indeed, this operation takes care both of the excess contribution as well as the contribution from embedded, distinguished varieties. Since curves in $|\mathscr{L}|$ with higher geometric genus must have strictly fewer than $r$ singular points, the corresponding distinguished varieties must be supported on the diagonal subspace $\bigcup_{\pi \in \Pi_{r}^{\circ}} \Delta^{(r)}_{\pi}$ of $X \times_{Y} \ldots \times_{Y} X.$ So the cycle class $\gamma_{r\ast} I_{r} \in A^{r}(Y)$ represents a cycle which is reduced and enumerates precisely the finite number of $r$-nodal curves in the generic subsystem $\mathbb{P}^{r},$ with an ordering of the $r$ nodes. Since there are $r!$ ways to order the $r$ nodes, the class $\gamma_{r\ast} I_{r}/r!$ enumerates $r$-nodal curves, i.e.,
\begin{equation}
N_{r}(S, \mathscr{L})H^{r} = \frac{1}{r!}\gamma_{r\ast} I_{r}.
\end{equation}

Since we defined $a_{1}(S,\mathscr{L})$ as $\int_{Y}\gamma_{1\ast}\xi,$ the pushdown to $Y$ of $\prod_{i=1}^{r}p^{\ast}_{i}\xi$ becomes $a_{1}(S,\mathscr{L})^{r}H^{r}.$ Also, Proposition \ref{prop:splitting} implies that $n^{(r)}_{\pi}B^{(r)}_{\pi}$ pushes down to $\prod_{i=1}^{r} a_{i}(S,\mathscr{L})^{s_{i}(\pi)}H^{r},$ with $s_{i}(\pi)$ denoting the number of blocks of size $i$ in the partition $\pi \in \Pi_{r}.$ For any $r$-tuple of non-negative integers $j_{i}$ such that $j_{1} + 2j_{2} + \ldots + rj_{r} = r,$ let $\widetilde{e}_{j_{1}\ldots j_{r}}$ denote the number of polydiagonals with $j_{i}$ blocks of size $i.$ Then it is clear that
\begin{equation}
N_{r}(S,\mathscr{L}) = \frac{1}{r!} \sum_{j_{1} + \ldots + rj_{r} = r} \widetilde{e}_{j_{1}\ldots j_{r}}\prod_{l=1}^{r} a_{l}(S,\mathscr{L})^{j_{l}}.
\end{equation}
Set $L_{r}(a_{1}(S,\mathscr{L}), \ldots, a_{r}(S,\mathscr{L}))$ to be the sum $\sum_{j_{1} + \ldots + rj_{r} = r} \widetilde{e}_{j_{1}\ldots j_{r}}\prod_{l=1}^{r} a_{l}(S,\mathscr{L})^{j_{l}}.$ If we regroup the polydiagonals by their number of blocks, $i,$ and note that polydiagonals with $i$ blocks can have no blocks of size $> r-i+1$ (indeed, each block must have at least one element, so we would get a number of elements $> (i-1) \cdot 1 + r-i+1 = r,$ which is impossible), then
\begin{displaymath}
L_{r}(a_{1}(S,\mathscr{L}), \ldots, a_{r}(S,\mathscr{L}))  = \sum_{i=1}^{r} \sum_{J_{r,i}} e_{j_{1}\ldots j_{r-i+1}} \prod_{l=1}^{r-i+1} a_{l}(S,\mathscr{L})^{j_{l}}.
\end{displaymath}
Here, $J_{r,i}$ is the set of tuples $(j_{1}, \ldots, j_{r-i+1})$ such that we have $\sum lj_{l} = r$ and $\sum j_{l} = i$ (so $\sum j_{l}$ is the number of blocks and $\sum lj_{l}$ is the number of elements for the corresponding partition). The coefficient $e_{j_{1} \ldots j_{r-i+1}}$ is the number of polydiagonals with $i$ blocks, of which $j_{l}$ have size $l.$

But, according to Definition \ref{def:part_bell}, this is exactly how the coefficients of the partial Bell polynomials are defined, so $L_{r}(a_{1}(S,\mathscr{L}), \ldots, a_{r}(S,\mathscr{L}))$ is in fact equal to the $r$th complete Bell polynomial $P_{r}$ in the $a_{i}(S,\mathscr{L}), 1 \leq i \leq r,$ which is what we wanted to prove.
\end{proof}

\section{On the equivalence of the polydiagonals}\label{sec:equivalences}
The previous section established the shape of the node polynomials $Z_r$, but is computationally incomplete, since apart from providing an intersection theoretical definition of the $a_i,$ it does not present them as linear combinations (with coefficients which are integers) of the Chern numbers of $(S,\mathscr{L}).$ The distinguished varieties supported on the small diagonal $\Delta^{(r)}_{12\ldots}$ of $X^{\times r}$ include the diagonal itself, in addition to embedded components.

Our approach here is to first consider the proper contribution of the polydiagonals, the objective being to compute the excess contribution from their union, $\Delta(r),$ to the intersection product $X_1 \cdot \ldots \cdot X_r.$ In the next section, we treat the residual contribution coming from embedded components.

We recall the definition of the Segre class of a closed subscheme:
\begin{definition}\label{def:segre}
Let $X$ be a closed subscheme of a scheme $Y.$ Let $C$ denote the normal cone of $X$ in $Y,$ and consider the projective completion
\begin{equation}
P(C\oplus \mathbf{1}) := \textnormal{Proj}(S^{\bullet}[z]).
\end{equation}
Denote by $q$ the projection from $P(C\oplus \mathbf{1})$ to $X,$ and by $\mathscr{O}(1)$ the canonical line bundle on $P(C \oplus \mathbf{1}).$ The Segre class of $X$ in $Y$ is the following class:
\begin{equation}
s(C) := q_{\ast}\left(\sum_{i \geq 0} c_{1}(\mathscr{O}(1))^{i} \cap [P(C \oplus \mathbf{1})]\right) \in A_{\ast}(X).
\end{equation} \hfill $\blacksquare$
\end{definition}

By \cite[Proposition 9.1.1]{Ful}, the equivalence of $\Delta(r)$ for the intersection product $X_1 \cdot \ldots \cdot X_r$ is 
\begin{equation}\label{eqn:diagonals_excess}
(X_{1} \cdot \ldots \cdot X_{r})^{\Delta(r)} = \left\{ \prod_{i=1}^{r}c(N_{X_{i}}F^{\times r}|\Delta(r)) \cap s(\Delta(r),F^{\times r})\right\}_{N-r}
\end{equation}
The structure of the subscheme $\Delta(r),$ however, makes any direct attempt to control this difficult. Indeed, $\Delta(r)$ has several irreducible components, and while one can compute the contribution of each $\Delta^{(r)}_{\pi}$ separately (see below), this does not directly yield the contribution of their union. To clarify this, we proceed in several steps:

\begin{definition}\label{def:equivalence}
For each $\pi \in \Pi_{r}^{\circ},$ denote by $\mathscr{E}^{(r)}_{\pi}$ the equivalence of $\Delta^{(r)}_{\pi}$ for the intersection product $X_{1} \cdot \ldots \cdot X_{r},$ that is
\begin{equation}\label{eqn:equivalence}
\mathscr{E}^{(r)}_{\pi} := \left(X_{1} \cdot \ldots \cdot X_{r}\right)^{\Delta^{(r)}_{\pi}} \in A_{N-r}(\Delta^{(r)}_{\pi}).
\end{equation}
Also, let $Q^{(r)}_{\pi}$ denote the integer $$\int_{Y} f^{(r)}_{\pi\ast} \left(X_{1} \cdot \ldots \cdot X_{r} \right)^{\Delta^{(r)}_{\pi}} \in \mathbb{Z},$$ where $f^{(r)}_{\pi}: \Delta^{(r)}_{\pi} \rightarrow Y$ is the composition of the embedding of $\Delta^{(r)}_{\pi}$ into $F^{\times r}$ and the projection $\gamma_{r} F^{\times r} \rightarrow Y.$

For each $r \geq 1,$ let $\mathscr{E}_{r}$ denote $\mathscr{E}^{(r)}_{12\ldots r},$ and set $Q_{r} := Q^{(r)}_{12\ldots r}.$ \hfill $\blacksquare$
\end{definition}

Below, we will compute the numbers $Q_{r}.$ For now, we note that they are --- in large part --- all we need to understand the equivalence of $\Delta(r):$

\begin{theorem}\label{thm:split_equiv}
Let $\pi \in \Pi_{r}^{\circ}.$ For each $i \in [r],$ let $s_{i}(\pi)$ denote the number of blocks of length $i$ in the partition $\pi.$ Then $$Q^{(r)}_{\pi} = \prod_{i=1}^{r} Q_{i}^{s_{i}(\pi)}.$$
\end{theorem}

\begin{proof}
On $F^{\times r},$ let $N^{(r)}_{i}$ denote $p_{i}^{(r)\ast}\mathscr{P}_{F/Y}(\mathscr{L} \boxtimes \mathscr{O}_{Y}(1)),$ where $p^{(r)}_{i}: F^{\times r} \rightarrow F$ are the projections. Also, let $\mathbb{B}(\pi)$ denote the set of blocks of the partition $\pi,$ and for $I \in \mathbb{B}(\pi),$ let $|I|$ denote the number of elements in $I$ and $p_I: F^{\times r} \rightarrow F^{\times |I|}$ the projection $\prod_{i \in I}p^{(r)}_{i}.$ Then we have:
\footnotesize
\begin{eqnarray*}
\mathscr{E}^{(r)}_{\pi} & = & (X_1 \cdot \ldots \cdot X_r)^{\Delta^{(r)}_{\pi}} \\
& = & \left\{c(N^{(r)}_1 \oplus \ldots \oplus N^{(r)}_r|\Delta^{(r)}_{\pi}) \cap s(\Delta^{(r)}_{\pi},F^{\times r})\right\}_{N-r} \\
& = & \left\{c\left(\bigoplus_{I \in \mathbb{B}(\pi)} p_{I}^{\ast}(N^{(|I|)}_{1} \oplus \ldots \oplus N^{(|I|)}_{|I|})|\Delta^{(r)}_{\pi} \right) \cap s\left(\prod_{I \in \mathbb{B}(\pi)} \Delta^{(|I|)}_{1\ldots |I|}, \prod_{I \in \mathbb{B}(\pi)} F^{\times |I|} \right) \right\}_{N-r} \\
& = & \prod_{I \in \mathbb{B}(\pi)} p_{I}^{\ast}\left\{c(N^{(|I|)}_{1} \oplus \ldots \oplus N^{(|I|)}_{|I|}|\Delta^{(|I|)}_{1\ldots |I|}) \cap s(\Delta^{(|I|)}_{1\ldots |I|},F^{\times |I|})\right\}_{N-|I|} \\
& = & \prod_{I \in \mathbb{B}(\pi)} p_{I}^{\ast}\mathscr{E}_{|I|},
\end{eqnarray*}
\normalsize
since $\Delta^{(r)}_{\pi} \cong \prod_{I \in \mathbb{B}(\pi)} \Delta^{(|I|)}_{1\ldots |I|}$ (fibered product over $Y$) and by definition of the intersection product as $A^{\ast}(F^{\times r}) \otimes A^{\ast}(F^{\times r}) \stackrel{\times}{\rightarrow} A^{\ast}(F^{\times r} \times F^{\times r}) \stackrel{\delta^{\ast}}{\rightarrow} A^{\ast}(F^{\times r}).$ But by a reasoning similar to Proposition \ref{prop:splitting}, the pushdown of $\prod_{I \in \mathbb{B}(\pi)} p_{I}^{\ast}\mathscr{E}_{|I|}$ to $Y$ is equal to $\prod_{I \in \mathbb{B}(\pi)} f_{\ast} \mathscr{E}_{|I|},$ hence
\begin{eqnarray*}
Q^{(r)}_{\pi} & = & \int_{Y} \prod_{I \in \mathbb{B}(\pi)} f_{\ast} \mathscr{E}_{|I|} \\
& = & \prod_{I \in \mathbb{B}(\pi)} Q_{|I|} = \prod_{i=1}^{r} Q_{i}^{s_{i}(\pi)},
\end{eqnarray*}
as claimed.
\end{proof}

At this point, the naive way to proceed would be to use the principle of inclusion-exclusion to express $(X_1 \cdot \ldots \cdot X_r)^{\Delta(r)}$ as a linear combination of the $\mathscr{E}^{(r)}_{\pi}.$ The following example illustrates that this is impossible:

\begin{example}\label{ex:case_three}
Let $r :=3.$ There are four diagonals to consider, the small diagonal $\Delta^{(3)}_{123}$ and the three large diagonals; $\Delta^{(3)}_{12|3}, \Delta^{(3)}_{13|2}$ and $\Delta^{(3)}_{23|1}.$ Each of those contains the small diagonal. Thus, the principle of inclusion-exclusion predicts the following equality (where the terms on the right hand side are pushed forward to $\Delta(3)$):
\begin{equation}\label{eqn:inc_exc_fail}
(X_{1} \cdot X_{2} \cdot X_{3})^{\Delta(3)} = \sum_{i \neq j}(X_{1} \cdot X_{2} \cdot X_{3})^{\Delta^{(3)}_{ij}} - 2(X_{1} \cdot X_{2} \cdot X_{3})^{\Delta^{(3)}_{123}}.
\end{equation}
When pushing this down to $Y$ and taking the degree, the right hand side becomes $3Q_1Q_2 - 2Q_3,$ because of Theorem \ref{thm:split_equiv}. But this is not the correct ``total'' equivalence, simply because Segre classes do not satisfy the principle of inclusion-exclusion. This failure is easily illustrated by considering the following example: Let $X$ be the subscheme of $\mathbb{P}^{2}$ defined as the union of two lines; since it is a divisor of degree 2, its Segre class in $\mathbb{P}^{2}$ is $2l-4l^2,$ with $l$ the class of a hyperplane. However, inclusion-exclusion predicts $(l-l^2)+(l-l^2)-l^2 = 2l-3l^2,$ which is wrong.

For more on this problem and how to understand it, see \cite{Alu1}. \hfill $\blacksquare$
\end{example}

For us, this means that we need to construct appropriate correction terms, $\mathscr{C}^{(r)}_{\pi},$ such that $(X_1 \cdot \ldots \cdot X_{r})^{\Delta(r)}$ can be written as a linear combination, not of the $\mathscr{E}^{(r)}_{\pi},$ but of the corrected terms $\mathscr{E}^{(r)}_{\pi} + \mathscr{C}^{(r)}_{\pi}.$ For this, we make use of the classical theory of multiple point formulas, following essentially Kleiman's \cite{Klei1}.

Let $f: X \rightarrow Y$ be the composition of the embedding $\iota$ of $X$ in $F = S \times Y$ and the projection $\gamma_{1}$ to $Y = |\mathscr{L}|.$ This is an lci of codimension 1. Its strict double points are points in $X$ corresponding to binodal curves with one marked node, while the double point locus also includes cuspidal curves. The double point formula \cite[Theorem 5.6]{Klei1} states that
\begin{eqnarray*}
m_{2} & = & f^{\ast}f_{\ast}[X] - c_{1} \cap [X] \\
& = & p_{1\ast} (p_{1}^{\ast}[X] \cdot p_{2}^{\ast}[X]) - \left\{\frac{f^{\ast}(c(T_{Y}))}{c(T_{X})}\right\}^{1} \cap [X] \\
& = & p_{1\ast} (p_{1}^{\ast}[X] \cdot p_{2}^{\ast}[X]) - \left\{c(N_{X}F)\iota^{\ast}\nu^{\ast}c(T_{S})^{-1} \cap [X] \right\}_{N-2}. \\
\end{eqnarray*}
Indeed, to show that $f^{\ast}f_{\ast}[X] = p_{1\ast}(p_{1}^{\ast}[X] \cdot p_{2}^{\ast}[X]),$ consider the fibre diagram
\[
\xymatrix
{
X \times_{Y} X \ar[d]_{p_1} \ar[r] & X \times X \ar[d]^{1 \times f} \\
X \ar[r]_{\gamma_f} \ar[d]_{f} & X \times Y \ar[d]^{f \times 1} \\
Y \ar[r]_{\delta} & Y \times Y
}
\]
where $\gamma_f: X \hookrightarrow X \times Y$ is the graph embedding of $X,$ and $\delta: Y \hookrightarrow Y \times Y$ is the diagonal embedding of $Y.$ Then
\begin{eqnarray*}
f^{\ast}f_{\ast}[X] & = & \gamma_{f}^{\ast}([X] \times f_{\ast}[X]) \\
& = & p_{1\ast}(\gamma_{f}^{!}[X \times X]) \\
& = & p_{1\ast}(\delta^{!}[X \times X]) = p_{1\ast}(p_{1}^{\ast}[X] \cdot p_{2}^{\ast}[X]), 
\end{eqnarray*}
where $\gamma_{f}^{!}: A^{\ast}(X \times X) \rightarrow A^{\ast}(X \times_{Y} X)$ and $\delta^{!}: A^{\ast}(X \times X) \rightarrow A^{\ast}(X \times_{Y} X)$ are the refined Gysin pullback homomorphisms induced by $\gamma_f$ and $\delta$ (see \cite[Sec. 6.2]{Ful} for a formal definition). On the other hand, to show that
\begin{displaymath}
\left\{\frac{f^{\ast}(c(T_{Y}))}{c(T_{X})}\right\}^{1} \cap [X] = \left\{c(N_{X}F)\iota^{\ast}\nu^{\ast}c(T_{S})^{-1} \cap [X] \right\}_{N-2},
\end{displaymath}
we simply use the standard exact sequence of the regular embedding $\iota,$
\begin{equation}\label{eqn:exact_seq}
0 \rightarrow T_{X} \rightarrow \iota^{\ast}T_{F} \rightarrow N_{X}F \rightarrow 0,
\end{equation}
and the fact that $c(T_{F}) = \nu^{\ast}c(T_{S}) \oplus \gamma_{1}^{\ast} c(T_{Y})$ when $\nu$ is the projection $F = S \times Y \rightarrow S.$ We conclude that $\int_{Y}f_{\ast}m_{2} = Q_{1}^{2} - Q_{2},$ where the $Q_{i}$ are the terms introduced in Definition \ref{def:equivalence}. Indeed,
\begin{eqnarray*}
\mathscr{E}_{2} & = & \left(X_{1} \cdot X_{2}\right)^{\Delta^{(2)}_{12}} \\
& = & \left\{p_1^{\ast}c(N_{X}F|\Delta^{(2)}_{12})p_2^{\ast}c(N_{X}F|\Delta^{(2)}_{12}) \cap s(\Delta^{(2)}_{12},F^{\times 2}) \right\}_{N-2} \\
& = & \left\{c(N_{X}F)^{2} \cap c(N_{X}F)^{-1} \iota^{\ast}\nu^{\ast}c(T_{S})^{-1} \cap [X] \right\}_{N-2} \\
& = & \left\{c(N_{X}F)\iota^{\ast}\nu^{\ast}c(T_{S})^{-1} \cap [X] \right\}_{N-2} \\
& = & c_{1} \cap [X].
\end{eqnarray*}

Now, the triple point formula \cite[Theorem 5.9]{Klei1} can be manipulated as follows (all maps have codimension 1):
\begin{eqnarray*}
m_{3} & = & f^{\ast}f_{\ast}m_2 - 2c_1 \cap m_2 + 2c_2 \cap m_1 \\
& = & f^{\ast}f_{\ast}(f^{\ast}f_{\ast}[X] - c_{1} \cap [X]) - 2c_{1} \cap (f^{\ast}f_{\ast}[X] - c_{1} \cap [X]) + 2c_{2} \cap [X] \\
& = & p_{1\ast} (p_{1}^{\ast}[X] \cdot p_{2}^{\ast}[X] \cdot p_{3}^{\ast}[X]) - 3c_{1} f^{\ast}f_{\ast}[X] + 2c_{1}^{2} \cap [X] + 2c_{2} \cap [X].
\end{eqnarray*}

We now rewrite the term $c_1^{2} \cap [X]:$ Let $c(f)(t) := 1 + \sum c_i t^i$ denote the polynomial $c_{t}(N_{X}F)\iota^{\ast}\nu^{\ast}c_{t}(T_{S})^{-1}$ in $t.$ Then $c_{1}$ is the coefficient of $t$ in $c(f)(t).$ On the other hand, the equivalence $\mathscr{E}_{3}$ is defined as
\begin{equation}
\left\{c(N_{X}F)^2\iota^{\ast}\nu^{\ast}c(T_{S})^{-2} \cap [X] \right\}_{N-3},
\end{equation}
which corresponds to capping the coefficient of $t^2$ in $c(f)(t)^2,$ namely $c_1^2 + 2c_2,$ with $[X].$ Thus,
\begin{equation}
2c_{1}^{2} \cap [X] = 2\mathscr{E}_{3} - 4c_{2} \cap [X].
\end{equation}

\begin{definition}
We denote by $C_3$ the integer
\begin{displaymath}
C_{3} := -\int_{Y} f_{\ast} (c_{2} \cap [X]) = -\int_{Y} f_{\ast} \left\{c(N_{X}F)\iota^{\ast}\nu^{\ast}c(T_{S})^{-1} \cap [X] \right\}_{N-3}.
\end{displaymath}
\hfill $\blacksquare$
\end{definition}

Recall that the third complete Bell polynomial is defined as $P_{3}(x_1,x_2,x_3) := x_1^3 + 3x_1x_2 + x_3.$ We therefore see that
\begin{equation}
\int_{Y}f_{\ast}m_{3} = Q_{1}^{3} - 3Q_{1}Q_{2} + 2Q_{3} + 2C_{3} = P_{3}(Q_{1},-Q_{2},2(Q_{3}+C_{3})).
\end{equation}
Note that, comparing with the original expression $Q_{1}^{3} - 3Q_{1}Q_{2} + 2Q_{3}$ predicted by inclusion-exclusion (cf. Example \ref{ex:case_three}), we recover a ``correction term.''

\begin{definition}
We denote by $C_{4}$ the integer
\begin{eqnarray*}
C_{4} & := & -\int_{Y} f_{\ast} \left(\frac{3}{2} \left\{c(N_{X}F)^2(\iota^{\ast}\nu^{\ast}c(T_{S})^{-1})^2 \cap [X] \right\}_{N-4}\right) \\
& & + \int_{Y} f_{\ast}\left( 2\left\{c(N_{X}F)\iota^{\ast}\nu^{\ast}c(T_{S})^{-1} \cap [X] \right\}_{N-4}\right).
\end{eqnarray*}
\hfill $\blacksquare$
\end{definition}

Recall that $P_{4}(x_{1},x_{2},x_{3},x_{4}) := x_1^4 + 6x_1^2x_2 + 4x_1x_3 + 3x_2^2 + x_4$ is the fourth Bell polynomial. Kleiman's 4-point formula \cite[Theorem 5.10]{Klei1} gives, by expanding the terms $m_2$ and $m_3,$
\begin{eqnarray*}
m_{4} & = & f^{\ast}f_{\ast}m_{3} - 3c_{1} \cap m_{3} + 6c_{2} \cap m_{2} - 6(c_{1}c_{2} + 2c_{3}) \cap m_{1} \\
& = & (f^{\ast}f_{\ast})^{3}[X] - 6c_{1} \cap (f^{\ast}f_{\ast})^{2}[X] + 8(c_{1}^{2}+c_{2}) \cap f^{\ast}f_{\ast}[X] \\ 
& & + 3c_{1} \cap f^{\ast}f_{\ast}(c_{1} \cap [X]) - 6(3c_{1}c_{2} + 2c_{3} +c_{1}^{3}) \cap [X].
\end{eqnarray*}
Now, $\mathscr{E}_{4}$ is defined as
\begin{equation}
\left\{c(N_{X}F)^3\iota^{\ast}\nu^{\ast}c(T_{S})^{-3} \cap [X]\right\}_{N-4},
\end{equation}
which corresponds to capping the coefficient of $t^3$ in $c(f)(t)^3,$ which is $c_{1}^3 + 6c_1c_2 + 3c_3,$ with $[X].$ Also, considering the terms appearing in the definition of $C_4,$ we have
\begin{equation}
\left\{c(N_{X}F)^2(\iota^{\ast}\nu^{\ast}c(T_{S})^{-1})^2 \cap [X] \right\}_{N-4}
\end{equation}
which corresponds to taking the coefficient of $t^3$ in $c(f)(t)^2,$ namely $2(c_1c_2 + c_3),$ and capping with $[X].$ Finally,
\begin{equation}
\left\{c(N_{X}F)(\iota^{\ast}\nu^{\ast}c(T_{S})^{-1}) \cap [X] \right\}_{N-4}
\end{equation}
corresponds to capping the coefficient of $t^3$ in $c(f)(t),$ namely $c_3,$ with $[X].$ Hence,
\begin{eqnarray*}
& & \int_{Y} -6 f_{\ast}((3c_{1}c_{2} + 2c_{3} +c_{1}^{3}) \cap [X])\\
& = & \int_{Y} -6 f_{\ast}((c_{1}^{3} + 6c_1c_2 + 3c_3 - 3/2 (2c_1c_2 + 2c_3) +2c_3) \cap [X]) \\
& = & -6(Q_{4}+C_{4}),
\end{eqnarray*}
and we see that
\begin{equation}
\int_{Y}f_{\ast}m_{4} = P_{4}(Q_1, -Q_2, 2(Q_3+C_3), -6(Q_4+C_4)).
\end{equation}

\begin{remark}\label{rem:multiple_bell}
There are two interesting observations to be made: First, we see that by combining certain terms in Kleiman's $r$-point formulas, we can express these formulas using Bell polynomials. Second, the ``correction terms'' $C_i,$ which a priori occur because we are trying to do inclusion-exclusion using objects (Segre classes) which do not behave well in this regard, are defined using the same classes which define the $Q_i,$ but considering parts of different dimensions. To state this more clearly, we introduce the class
\begin{equation}
\boxed{M_{r}(S,\mathscr{L}) := c(N_{X}F)^{r-1}(\iota^{\ast}\nu^{\ast}c(T_{S}))^{r-1} \cap [X]}
\end{equation}
for each $r \geq 2.$ Then $Q_{r}$ is obtained from the component of $M_{r}(S,\mathscr{L})$ of dimension $N-r,$ while we have
\begin{eqnarray*}
C_{3} & = & -\int_{Y} f_{\ast}\left\{M_{2}(S,\mathscr{L})\right\}_{N-3}, \\
C_{4} & = & -\int_{Y} f_{\ast} \left(3/2\left\{M_{3}(S,\mathscr{L})\right\}_{N-4} - 2\left\{M_{2}(S,\mathscr{L})\right\}_{N-4}\right).
\end{eqnarray*}
\end{remark}

We see this as evidence supporting the following conjecture (recall that $P_r$ denotes the $r$th complete Bell polynomial in $r$ variables):

\begin{conjecture}\label{conj:linearity_corr}
For $r \geq 1,$ there is a $\mathbb{Q}$-linear combination $C_{r}$ of the integers
\begin{equation}
\int_{Y} f_{\ast} \{M_{i}(S,\mathscr{L})\}_{N-r},
\end{equation}
for $2 \leq i \leq r-1,$ with $C_1 = C_2 = 0,$ such that
\begin{equation}
\int_{Y} f_{\ast}m_{r} = P_{r}(Q_1 + C_1, -2(Q_2 + C_2), \ldots, (-1)^{r-1}(r-1)!(Q_r+C_r)).
\end{equation}
\end{conjecture}

Our next aim is to compute the equivalence terms $Q_n$ in the case of the projective plane; this simplification allows for a clearer presentation, but it is not difficult to see that more generally, both the equivalence terms $Q_n$ and (at least for $n \leq 4$) the correction terms $C_n$ are linear combinations of the four Chern numbers of $(S,\mathscr{L}),$ and the general closed expressions for the $Q_n$ can be obtained following the same steps as below, although the computations are slightly more involved.

Let $S := \mathbb{P}^{2}$ and $\mathscr{L} := \mathscr{O}_{\mathbb{P}^{2}}(d).$ By Lemma \ref{prop:isom}, we know that $X$ is regularly embedded in $F$ with normal bundle
\begin{equation}\label{eqn:principal}
N_{X}F \cong \iota^{\ast}\mathscr{P}^{1}_{F/Y}\Bigl(\mathscr{O}_{\mathbb{P}^{2}}(d) \boxtimes \mathscr{O}_{Y}(1)\Bigr),
\end{equation}
and $[X] = c_3(N_{X}F).$

For a regular embedding $X \hookrightarrow Y$ we have $s(X,Y) = c(N_{X}Y)^{-1} \cap [X]$ by \cite[Section 4.2]{Ful}. Now, the embedding of $\Delta^{(n)}_{12\ldots n}$ in $F^{\times n}$ splits as $$\Delta^{(n)}_{12\ldots n} \hookrightarrow F \stackrel{\delta_{n}}{\hookrightarrow} F^{\times n},$$
where $\delta_n$ is the diagonal embedding. Hence
\begin{eqnarray*}
Q_{n} & = & \int_{Y} f_{\ast} (X_{1} \cdot \ldots \cdot X_{n})^{\Delta^{(n)}_{1\ldots n}} \\
& = & \int_{Y} f_{\ast} \left\{\prod_{i=1}^{n} c(p_{i}^{\ast}N_{X}F|\Delta^{(n)}_{1\ldots n}) \cap c(N_{\Delta^{(n)}_{1\ldots n} }F^{\times n})^{-1} \cap [\Delta^{(n)}_{1\ldots n} ]\right\}_{N-n} \\
& = & \int_{Y} f_{\ast} \left\{c(N_{X}F)^{n} c(N_{X}F)^{-1}c(N_{F}F^{\times n})^{-1} \cap [X] \right\}_{N-n} \\
& = & \int_{Y} f_{\ast} \left\{c(N_{X}F)^{n-1}c\left(T_{F/Y}^{\oplus(n-1)}\right)^{-1} \cap [X] \right\}_{N-n} \\
& = & \int_{Y} f_{\ast} \left\{c(N_{X}F)^{n-1}c(T_{F/Y})^{-(n-1)} \cap [X] \right\}_{N-n}.
\end{eqnarray*}

Let $l$ denote the class of a hyperplane on $\mathbb{P}^{2},$ and $H$ the class of a hyperplane on $Y = |\mathscr{L}| = \mathbb{P}^{N}.$ So $l^{3}=0$ and $H^{N+1} = 0.$ It is well-known that $c(T_{\mathbb{P}^{2}})^{-1} = 1-3l+6l^2.$ Hence the computation of $Q_{n}$ reduces to finding the coefficient of $H^{n}l^{2}$ in the polynomial
\begin{equation}\label{eqn:poly}
M_{n}(l,H,d) := \bigl(1 + H + (d-1)l \bigr)^{3(n-1)}(1-3l+6l^{2})^{n-1}\bigl(H+(d-1)l\bigr)^{3}.
\end{equation}
For this, we first extract the coefficient of $H^{n};$ this is a polynomial in $l$ and $d,$ from which we extract the coefficient of $l^{2}.$ We have:
\[
\begin{dcases}
\Bigl(1 + H+ (d-1)l \Bigr)^{3(n-1)} = \sum_{k=0}^{3n-3} {3n-3 \choose k}H^{k}\Bigl(1+(d-1)l\Bigr)^{3(n-1)-k}; \\
\Bigl(H+(d-1)l\Bigr)^{3} = H^{3} + 3H^{2}(d-1)l + 3H(d-1)^{2}l^{2},
\end{dcases}
\]
\noindent since $l^{3} = 0.$ Therefore, the coefficient of $H^{n}$ is easily shown to be

\begin{displaymath}
(1-3l+6l^{2})^{n-1}(1+(d-1)l)^{2n-2}(x_{n}l^{2} + y_{n}l + z_{n}),
\end{displaymath}
where
\[
\begin{dcases}
x_{n} := 3(d-1)^{2} {3n-3 \choose n-1} + 3(d-1)^{2} {3n-3 \choose n-2} + {3n-3 \choose n-3}(d-1)^{2}, \\
y_{n} := 3(d-1){3n-3 \choose n-2} + 2(d-1){3n-3 \choose n-3}, \\
z_{n} := {3n-3 \choose n-3}.
\end{dcases}
\]
To find the coefficient of $l^{2}$ in this expression, expand
\begin{displaymath}
(1-3l+6l^{2})^{n-1} = \sum_{k=0}^{2} {n-1 \choose k}3^{k}l^{k}(2l-1)^{k},
\end{displaymath}
with the convention that ${n \choose k} =0$ if $k > n.$ This is equal to
\begin{displaymath}
\alpha_{n} := 1 -3(n-1)l+\left(6(n-1)+9 {n-1 \choose 2}\right)l^{2}.
\end{displaymath}
On the other hand, we get
\begin{displaymath}
\beta_{n} := (1+(d-1)l)^{2n-2} = 1 + (2n-2)(d-1)l + {2n-2 \choose 2}(d-1)^{2}l^{2}.
\end{displaymath}
So we are looking for the coefficient of $l^{2}$ in the expression $\alpha_{n}\beta_{n}(x_{n}l^{2}+y_{n}l+z_{n}),$ which is
\begin{eqnarray*}
& & \left(6(n-1)+9{n-1 \choose 2}\right)z_{n} + {2n-2 \choose 2}(d-1)^{2}z_{n} + x_{n} \\
& - & 3(n-1)(2n-2)(d-1)z_{n} -3(n-1)y_{n} + (2n-2)(d-1)y_{n}.
\end{eqnarray*}
To conclude, we have the following theorem:

\begin{theorem}\label{thm:equiv_diag}
In the case of $\mathbb{P}^{2},$ the equivalence of the small diagonal $\Delta_{12\ldots n}$ for the intersection product $X_{1} \cdot \ldots \cdot X_{n}$ is a quadratic polynomial in $d,$ namely
\begin{equation}\label{eqn:equiv_diag}
Q_{n} = f_{n}d^{2} + g_{n}d + h_{n},
\end{equation}
where (after some simplifications):
\[
\begin{dcases}
f_{n} :=  3{3n-3 \choose n-1} + 3{3n-3 \choose n-2}(2n-1) + n{3n-3 \choose n-3}(2n-1), \\
g_{n} :=  -2n{3n-3 \choose n-3}(5n-4) - 3{3n-3 \choose n-2}(7n-5) - 6{3n-3 \choose n-1}, \\
h_{n} :=  {3n-3 \choose n-3}\left(\frac{25}{2}n^{2}-\frac{29}{2}+3\right) + 3{3n-3 \choose n-2}(5n-4) + 3{3n-3 \choose n-1}.
\end{dcases}
\]
\end{theorem}

\begin{table}
\centering
\begin{tabular}{|l||l|l|}
\hline
$n$ & $Q_{n}$ & $C_{n}$ \\
\hline
1 & $3d^2-6d+3$ & 0 \\
2 & $18d^2-45d+27$ & 0 \\
3 & $150d^2-444d+315$ & $-(30d^2-96d+72)$ \\
4 & $1260d^2 -4140d + 3285$ & $-(420d^2 - 1425d + 1158)$\\
\hline
\end{tabular}
\caption{Equivalence and correction terms for $1 \leq n \leq 4.$}
\label{table:equivalences}
\end{table}

\begin{remark}\label{rem:combinations}
Above, we saw that the ``correction terms'' $C_{i}$ were linear combinations of terms which arose from the same polynomials $M_{n}(l,H,d),$ but extracting different coefficients. Of course, one can obtain closed formulas for these coefficients, proceeding the same way as above. For $1 \leq n \leq 4,$ the concrete expressions for $Q_{n}$ and $C_{n}$ are provided in Table \ref{table:equivalences}.
\end{remark}

\section{On the residual term}\label{sec:residual}
Recall that, up to a factor $(-1)^{i-1}(i-1)!,$ $a_{i}(S,\mathscr{L})$ was defined as the degree of the pushdown through $\gamma_1$ of $\iota_{\ast}B^{(i)}_{1\ldots i} \in  A^{i}(S \times Y),$ where $\iota$ denotes the inclusion $X \hookrightarrow F.$ In the previous section, we treated the contribution from the small diagonal $\Delta^{(i)}_{12\ldots i},$ while neglecting the contribution from embedded components, i.e., distinguished varieties having support inside this diagonal. Thus, the remaining question, which we explore in this section, is how the embedded components (the ``residual'' locus) contributes to $a_i.$

Assume that $\mathscr{L}$ is $r$-very ample, so that there is no interference from for instance non-reduced curves (cf. Theorem \ref{thm:shape}). The multiplicative structure imposed by the lattice of polydiagonals applies for the embedded components as well, so it suffices to study the embedded components with support on the small diagonal $\Delta^{(r)}_{12\ldots r}.$ We wish to show that the components supported on the small diagonal contribute linearly in the four Chern numbers of $(S,\mathscr{L});$ this is achieved, with the exception of one conjectural result (Conjecture \ref{conj:dep}). The geometric interpretation of the contribution is neither immediate nor easy, but is discussed towards the end of the section for low values of $r.$

Let $\Delta^{(r)}_{X}$ be the small diagonal in $X^{\times r}$ and $\Delta^{(r)}_{F}$ the small diagonal in $F^{\times r}.$ The arguments themselves are purely of technical nature. We proceed as follows: Let $V_{r}$ denote the blowup of $F^{\times r}$ along the small diagonal $\Delta^{(r)}_{F},$ and let $D_{r}$ be the exceptional divisor. We denote by $\widetilde{X_{i}}$ the strict transform of $X_{i}$ under the morphism $\pi_{r}: V_{r} \rightarrow F^{\times r}.$

Consider the subschemes $W_{r}$ and $W_{r}(X)$ of $V_{r}$ whose sheaves of ideals are
\begin{eqnarray*}
\mathscr{I}_{W_{r}} & := & \mathscr{I}_{D_{r}} \cdot \left(\mathscr{I}_{D_{r}} + \sum_{i=1}^{r} \mathscr{I}_{\widetilde{X_{i}}}\right) \\
\mathscr{I}_{W_{r}(X)} & := & \mathscr{I}_{\pi_{r}^{-1}(\Delta^{(r)}_{X})} \cdot \left(\mathscr{I}_{D_{r}} + \sum_{i=1}^{r} \mathscr{I}_{\widetilde{X_{i}}}\right).
\end{eqnarray*}
Then $W_{r}(X)$ is regularly embedded in $W_{r}$ with normal bundle
\begin{equation}\label{eqn:normal_bundle}
N_{W_{r}(X)}W_{r} \cong \eta_{r}^{\ast} N_{\Delta^{(r)}_{X}}\Delta^{(r)}_{F},
\end{equation}
with $\eta_{r}$ the restriction of $\pi_{r}$ to the small diagonal of $X^{\times r}.$ We may consider the residual scheme $\textnormal{Res}_{r}$ of the divisor $D_{r}$ in $W_{r}.$ Then, according to \cite[Propostion 9.2]{Ful}, we have for all $m \geq 0,$
\begin{equation}\label{eqn:relate_segre}
s(W_{r},V_{r})_{m} = s(D_{r},V_{r})_{m} + \mathscr{R}(r)_{m} \in A_{\ast}(D_{r}),
\end{equation}
where we have defined
\begin{equation}\label{eqn:define_res}
\mathscr{R}(r)_{m} := \sum_{j=0}^{N+2r-m}{N+2r-m \choose j}[-D_{r}]^{j}s(\textnormal{Res}_{r},V_{r})_{m+j}.
\end{equation}

It follows that the contribution to $X_{1} \cdot \ldots \cdot X_{r}$ from the small diagonal \textit{with} embedded components is
\footnotesize
\begin{eqnarray*}
\eta_{r\ast} (X_{1} \cdot \ldots \cdot X_{r})^{W_{r}(X)} & = & \eta_{r\ast}\left\{\prod_{i=1}^{r} c(\eta_{r}^{\ast}N_{i}) \cap s(W_{r}(X),V_{r})\right\}_{N-r} \\
& = & \eta_{r\ast}\left\{\prod_{i=1}^{r} c(\eta_{r}^{\ast}N_{i})c(\eta_{r}^{\ast}N_{\Delta^{(r)}_{X}}\Delta^{(r)}_{F})^{-1} \cap s(W_{r},V_{r})\right\}_{N-r} \\
& = & (X_{1} \cdot \ldots \cdot X_{r})^{\Delta^{(r)}_{X}} + \eta_{r\ast} \sum_{m \geq 0} \left\{c\left(\eta_{r}^{\ast}N_{\Delta^{(r)}_{X}}\Delta^{(r)}_{F}\right)^{r-1} \cap \mathscr{R}(r)_{m}\right\}_{N-r}
\end{eqnarray*}
\normalsize
where $N_{i}$ denotes the restriction to the small diagonal of the normal bundle of $X_{i}$ in $F^{\times r}.$ The last equality follows from Eq. (\ref{eqn:relate_segre}).

The following theorem, due to Keel, expresses the Chow ring of a blow-up. Let $V$ be a variety and let $i: U \hookrightarrow V$ be a regularly embedded subvariety of codimension $d.$ Denote by $N$ the normal bundle of $U$ in $V.$ Let $\pi: \widetilde{V} \rightarrow V$ be the blow-up of $V$ along $U$ and denote by $\widetilde{U}$ the exceptional divisor. Define $g$ and $j$ by the commutative diagram:
\[
\xymatrix
{
\widetilde{U} \ar^{j}[r] \ar_{g}[d] & \widetilde{V} \ar^{\pi}[d] \\
 U \ar^{i}[r] & V
}
\]
Let $P(t)$ be any polynomial whose constant term is $[U] \in A^{\ast}(V)$ and whose restriction to $A^{\ast}(U)$ is the Chern polynomial of the normal bundle $N,$ that is,
\begin{equation}\label{eqn:chern_poly}
i^{\ast}P(T) = t^{d} + t^{d-1}c_{1}(N) + \ldots + c_{d-1}(N)t + c_{d}(N).
\end{equation}
\begin{theorem}\label{thm:chow_blowup}
\emph{(\cite{Keel}, Theorem 1 of Appendix.)}
Suppose the map of bivariant rings $i^{\ast}: A^{\ast}(V) \rightarrow A^{\ast}(U)$ is surjective. Then $A^{\ast}(\widetilde{V})$ is isomorphic to
\begin{equation}\label{eqn:chow_iso}
\frac{A^{\ast}(V)[t]}{(P(t),t \cdot \textnormal{ker}(i^{\ast}))}.
\end{equation}
This isomorphism is induced by
\begin{displaymath}
\pi^{\ast}: A^{\ast}(V) \rightarrow A^{\ast}(\widetilde{V})
\end{displaymath}
and by sending $-t$ to the class of the exceptional divisor.
\end{theorem}

We use this theorem to describe the Chow ring of the blow-up $V_{r}:$

\begin{corollary}\label{cor:chow_conf}
The Chow ring of $V_{r}$ is
\begin{equation}\label{eqn:conf_iso}
A^{\ast}(V_{r}) = A^{\ast}(F^{\times r})[D_r]/I_{r},
\end{equation}
where $I_{r}$ is the ideal generated by the following elements:
\begin{enumerate}
\item all $[D_r] \cdot (p_{i}^{\ast}\alpha - p_{j}^{\ast}\alpha)$ for $\alpha \in A^{\ast}(F^{\times r});$
\item $J_r \cdot [D_r],$ where $J_r$ is the kernel of the restriction map $\delta_{r}^{\ast}: A^{\ast}(F^{\times r}) \rightarrow A^{\ast}(\Delta^{(r)}_{F});$
\item $P_{r}(-[D_r]),$ where $P_{r}(t) := t^{2r-2} + \sum_{i=1}^{2r-2}\nu^{\ast}c_{i}(T_{S}^{\oplus (r-1)})t^{2r-2-i},$ $\nu$ being the projection from $F = S \times Y$ to the surface $S.$
\end{enumerate}
\end{corollary}

\begin{proof}
This follows easily from Theorem \ref{thm:chow_blowup}, and using the fact that, in the Grothendieck ring,
\begin{equation}
N_{\Delta^{(r)}_{F}}F^{\times r} = T_{F/Y}^{\oplus (r-1)} = (\nu^{\ast}T_{S})^{\oplus (r-1)}.
\end{equation}
\end{proof}

Next, we describe the Chow ring of the divisor $D_{r}:$

\begin{proposition}\label{prop:chow_div}
For all $r \geq 2,$ the Chow ring $A^{\ast}(D_{r})$ is
\begin{equation}\label{eqn:iso_div}
A^{\ast}(D_{r}) = A^{\ast}(V_r)/K_r,
\end{equation}
where $K_r$ is the ideal generated by all $p_{i}^{\ast}\alpha - p_{j}^{\ast}\alpha$ for $i,j \in [r]$ and $\alpha \in A^{\ast}(F).$
\end{proposition}

\begin{proof}
 This follows from \cite[Corollary 7b]{FM}.
\end{proof}

The residual scheme $\textnormal{Res}_{r}$ is a subscheme of $V_{r}$ whose sheaf of ideals is
\begin{equation}\label{eqn:res_sheaf}
\mathscr{I}^{Res}_{r} := \mathscr{I}_{D_{r}} + \sum_{i=1}^{r} \mathscr{I}_{\widetilde{X_{i}}},
\end{equation}
i.e., it is the scheme-theoretic intersection $\textnormal{Res}_{r} = D_{r} \cap \bigcap_{i=1}^{r} \widetilde{X_{i}}.$ We introduce some notations: Let $L := c_{1}(\mathscr{L})$ and $K:= c_{1}(\mathscr{K}_{S}),$ which are classes in $A^{1}(S).$ The second Chern class of $S$ is denoted by $x \in A^{2}(S).$ Let $L,K$ and $x$ also denote their own pullbacks, through $\nu,$ to $F.$ Finally, let $H$ be the class of a hyperplane in $Y = \mathbb{P}^{N},$ and its pullback to $F.$ We consider $L,K,H$ to be weighted variables of degree 1, while $x$ is considered to have degree 2.

\begin{conjecture}\label{conj:dep}
The Segre class of $\textnormal{Res}_{r}$ in $V_{r},$ expressed in the Chow ring of $D_{r},$ is a polynomial in $L, K, H, [D_{r}]$ and $x.$
\end{conjecture}

\begin{remark}
By Proposition \ref{prop:isom}, $X$ is the zero scheme of a section $z'$ of the vector bundle $\mathscr{P}^{1}_{F/Y}(\widetilde{\mathscr{L}}),$ and is regularly embedded in $F$ of codimension 3 (and the Chern class of its normal bundle is a polynomial in $L,K,H$ and $x$). It follows that $X_{i}$ is the zero scheme of $p_{i}^{\ast}z',$ a section of $p_{i}^{\ast}\mathscr{P}^{1}_{F/Y}(\widetilde{\mathscr{L}}),$ and the strict transform $\widetilde{X_i}$ is the zero scheme of the induced section of
\begin{displaymath}
\pi^{\ast}p_i^{\ast}\mathscr{P}^{1}_{F/Y}(\widetilde{\mathscr{L}}) \otimes \mathscr{O}_{V_{r}}(-D_{r}).
\end{displaymath}
Hence, it seems plausible that the push-forward to $D_r$ of the Segre class of $\bigcap_{i=1}^{r} (\widetilde{X_i} \cap D_r)$ in $V_r$ is a function only of the pullbacks of $L,K,H,x$ through the projections $p_j,$ and $[D_r].$ Now, by Proposition \ref{prop:chow_div}, it follows that $p_i^{\ast}L = p_j^{\ast}L$ in $A^{\ast}(D_r),$ for all $i,j,$ and similarly for $K,H$ and $x.$ Hence, the push-forward of this Segre class to $D_r$ should be a polynomial in $L,K,H,x$ and $[D_r]$ only.
\end{remark}

The (push-forward of the) class $\eta_{r\ast}(X_{1} \cdot \ldots \cdot X_{r})^{W_{r}(X)}$ lives in $A^{\ast}(\Delta^{(r)}_{F}) \cong A^{\ast}(F),$ hence we make the following definition:

\begin{definition}\label{def:res_contrib}
For each $r \geq 2,$ define
\begin{equation}\label{eqn:res_contrib}
R_{r} := \sum_{m \geq 0}\int_{Y} f_{\ast} \eta_{r\ast} \left\{c(\eta_{r}^{\ast}N_{\Delta^{(r)}_{X}}\Delta^{(r)}_{F})^{r-1} \cap \mathscr{R}(r)_{m}\right\}_{N-r}.
\end{equation} \hfill $\blacksquare$
\end{definition}

This is the (degree of) the contribution from embedded components supported on $\Delta^{(r)}_{X}.$ The following statement is then a consequence of Conjecture \ref{conj:dep}:

\begin{conj-prop}\label{conj-prop:linearity}
There exists a linear polynomial $T^{\text Res}_{r}$ in four variables and with integer coefficients, such that
\begin{equation}\label{eqn:polynomiality}
R_{r} = T^{\text Res}_{r}(\partial,k,s,x),
\end{equation}
where $\partial := \mathscr{L}^{2}, k := \mathscr{LK}_{S},s:=\mathscr{K}_{S}^{2},x:=c_{2}(S)$ are the four Chern numbers of the pair $(S,\mathscr{L}).$
\end{conj-prop}

\begin{proof}
We are interested in the $(N-r)$-dimensional part of the class
\footnotesize
\begin{eqnarray*}
\sigma_{r} & := & \eta_{r\ast} \sum_{m \geq 0}c\left(\eta_{r}^{\ast}N_{\Delta^{(r)}_{X}}\Delta^{(r)}_{F}\right)^{r-1} \cap \mathscr{R}(r)_{m} = \sum_{m \geq 0} c\left(N_{\Delta^{(r)}_{X}}\Delta^{(r)}_{F}\right)^{r-1} \cap \eta_{r\ast}\mathscr{R}(r)_{m} \\
& = & \sum_{m \geq 0}\sum_{j=0}^{N+2r-m} (-1)^{j} {N+2r-m \choose j} c\left(N_{\Delta^{(r)}_{X}}\Delta^{(r)}_{F}\right)^{r-1} \cap \eta_{r\ast}\left([D_{r}]^{j} \cdot s_{m+j} \right) \in A_{\ast}(\Delta^{(r)}_{F}),
\end{eqnarray*}
\normalsize
where $s_{m+j}$ is the component of dimension $m+j$ of the Segre class $s(\textnormal{Res}_{r},V_{r}).$ By Conjecture \ref{conj:dep} and the fact that the exceptional divisor $D_{r}$ satisfies a polynomial equation involving $L,K,H$ and $x$ (cf. Corollary \ref{cor:chow_conf}), $\sigma_{r}$ is a polynomial in these four classes. Since the dimension of $\Delta^{(r)}_{F}$ is $N+2,$ the part of the class $\sigma_{r}$ of dimension $N-r$ is the part of this polynomial of total degree $N+2 - (N-r) = r+2.$ Pushing down to $Y$ and multiplying with $H^{N-r}$ kills everything but the part involving $H^{r},$ and we are left with a quadratic polynomial in $L,K$ and $x$ when $x$ is considered to have degree 2, i.e., a linear polynomial in $\partial,k,s$ and $x$ (when $x$ is considered to have degree 1).
\end{proof}

To summarize our results at this point, we have the following decomposition of $a_i(S,\mathscr{L}):$
\begin{equation}
a_{i}(S,\mathscr{L}) = (-1)^{i-1}(i-1)!(Q_{i} + C_{i} + R_{i}).
\end{equation}
The equivalence term $Q_i$ can be computed and given a closed formula, and is a linear combination (with coefficients which are integers) of the Chern numbers of $(S,\mathscr{L}).$ The correction term $C_i$ can a priori also be computed and shown to have the same behaviour (and for $1 \leq i \leq 4,$ this is a theorem by the previous section). The residual term $R_i$ is a linear combination of the four Chern numbers $\partial,k,s,x,$ provided Conjecture \ref{conj:dep} holds. Thus, we have to a large extent identified the $a_i(S,\mathscr{L}).$

Note that, as proposed in \cite[Theorem 2.1]{Qvi}, one can also use the G\"ottsche--Yau--Zaslow formula (cf. Conjecture \ref{conj:gen_got}) together with some power series manipulations to show that each $a_{i}(S,\mathscr{L})$ must have the desired behaviour, namely that for each $i \geq 1,$ the integer $a_{i}(S,\mathscr{L})$ defined above is the value taken on $(\partial,k,s,x)$ by a universal, linear polynomial in four variables with integer coefficients. It is convenient to denote these polynomials by $a_{i}(\partial,k,s,x).$ Hence, there exist sequences of integers $\{D_{i}\}_{i \geq 1}, \{E_{i}\}_{i \geq 1}, \{F_{i}\}_{i \geq 1}$ and $\{G_{i}\}_{i \geq 1}$ such that
\begin{equation}
a_{i}(\partial,k,s,x) = (-1)^{i-1}(i-1)!(D_{i}\partial + E_{i} k +F_{i}s + G_{i}x).
\end{equation}

One can even compute the polynomials $a_{i}(\partial,k,s,x)$ from the G\"ottsche--Yau--Zaslow formula, altough this depends on knowing the coefficients of the power series $B_{1}(q)$ and $B_{2}(q),$ which are still not well understood. G\"ottsche computed these power series up to degree 28, cf. \cite[Remark 2.5]{Got}, a computation which depends on the fact (recently proven by Kleiman--Shende in \cite{KS}) that plane $r$-nodal curves of degree $d$ are enumerated by universal polynomials when $r \leq 2d-2.$ The algorithm for extracting the $a_i$ from the G\"ottsche--Yau--Zaslow formula is presented in \cite[Algorithm 2.1]{Qvi}; its output is collected in Table \ref{table:polys} for $1 \leq i \leq 15.$ The polynomials $\widetilde{a}_{i}(\partial, k, s, x)$ are obtained by dividing $a_{i}(\partial, k, s, x)$ by $(i-1)!.$

\begin{table}
\centering
\begin{sideways}
\scriptsize
\begin{tabular}{|l|l|}
\hline
$a_{1} =$ & $3\partial +  2k +  x$ \\
\hline
$a_{2} =$ & -- 42$\partial$ -- 39k -- 6s -- 7x\\
\hline
$a_{3} =$ & 1380$\partial$ + 1576k + 376s + 138x\\
\hline
$a_{4} =$ & --72360$\partial$ --95670k -- 28842s --3888x \\
\hline
$a_{5} =$ & 5225472$\partial$ +  7725168k + 2723400s + 84384x\\
\hline
$a_{6} =$ & -- 481239360$\partial$ -- 778065120k -- 308078520s + 7918560x\\
\hline
$a_{7} =$ & 53917151040$\partial$ + 93895251840k + 40747613760s -- 2465471520x\\
\hline
$a_{8} =$ & -- 7118400139200$\partial$ -- 13206119880240k -- 6179605765200s  + 516524964480x \\
\hline
$a_{9} =$ & 1082298739737600$\partial$ + 2121324101971200k + 1057994510106240s -- 105531591674880x  \\
\hline
$a_{10}=$ & -- 186244876934645760$\partial$ -- 383178257123397120k -- 201938068481143680s + 22522077486397440x\\
\hline
$a_{11}=$ & 35785074342095769600$\partial$ + 76882882686451430400k + 42529950621208512000s -- 5120189378609356800x\\
\hline
$a_{12}=$ & -- 7593954156671416934400$\partial$ -- 16965814444711292160000k -- 9799242960045675628800s + 
 1246637955659688345600x \\
\hline
$a_{13}=$ & 1764002599954269954048000$\partial$ + 4083791314361072077209600k + 2452287375661994231961600s -- 325131495890223904358400x\\
\hline
$a_{14}=$ & --445196702136181894778880000$\partial$ -- 1064857909823340069685248000k -- 662444750461765046378803200s + 90666752530924449021542400x \\
\hline
$a_{15}=$ & 121304301227469541054089216000$\partial$ + 299017798634897453079185817600k + 192137539658526071385289113600s -- 26963216698297962471175987200x \\
\hline
\hline
$\widetilde{a}_{1} =$ & 3$\partial$ +  2k +  x \\
\hline
$\widetilde{a}_{2} =$ & --42$\partial$ -- 39k -- 6s -- 7x\\
\hline
$\widetilde{a}_{3} =$ & 690$\partial$ + 788k + 188s + 69x\\
\hline
$\widetilde{a}_{4} =$ & --12060$\partial$ -- 15945k -- 4807s -- 648x \\
\hline
$\widetilde{a}_{5} =$ & 217728$\partial$ +  321882k + 113475s + 3516x\\
\hline
$\widetilde{a}_{6} =$ & -- 4010328$\partial$ -- 6483876k -- 2567321s + 65988x\\
\hline
$\widetilde{a}_{7} =$ & 74884932$\partial$ + 130410072k + 56593908s -- 3424266x\\
\hline
$\widetilde{a}_{8} =$ & --1412380980$\partial$ -- 2620261881k -- 1226112255s  + 102485112x \\
\hline
$\widetilde{a}_{9} =$ & 26842726680$\partial$ + 52612204910k + 26239943207s -- 2617350984x  \\
\hline
$\widetilde{a}_{10}=$ & -- 513240952752$\partial$ -- 1055936555124k -- 556487181661s + 62064807888x\\
\hline
$\widetilde{a}_{11}=$ & 9861407170992$\partial$ + 21186861410508k + 11720114258490s -- 1410986931936x\\
\hline
$\widetilde{a}_{12}=$ & --190244562607008$\partial$ -- 425029422316200k -- 245491696730341s + 31230909182592x \\
\hline
$\widetilde{a}_{13}=$ & 3682665360521280$\partial$ + 8525631885908256k + 5119580760611226s -- 678769122880224x \\
\hline
$\widetilde{a}_{14}=$ & --71494333556133600$\partial$ -- 171005998538392560k  -- 106382292871378404s -- 14560213534363728x\\
\hline
$\widetilde{a}_{15}=$ & 1391450779290676680$\partial$ + 3429957097334083248k + 2203960837196658328s -- 309288199242633956x\\
\hline
\hline
\end{tabular}
\end{sideways}
\caption{The polynomials $a_{i}(\partial,s,k,x).$}
\label{table:polys}
\end{table}
\normalsize

Now, inverting the argument, both $B_{1}(q)$ and $B_{2}(q)$ can be deduced from the $a_i.$ Applying the G\"ottsche--Yau--Zaslow formula for an algebraic surface $S$ with $\chi(\mathscr{O}_{S})=0$ (and therefore with $x=-s$), and with $\mathscr{L}$ trivial, we get
\begin{equation}
\sum_{r \geq 0} Z_{r}(0,0,s,-s)(DG_{2}(\tau))^{r} = B_{1}(q)^{s}.
\end{equation}
Assume $B_{1}(q) = \sum_{r=0}^{\infty}b^{(1)}_{r}q^{r},$ and $\log B_{1}(q) = \sum_{r=1}^{\infty} c^{(1)}_{r}q^{r}.$ Then, by the definition of Bell polynomials (cf. Eq. (\ref{eqn:formal_id})),
\begin{equation}
b^{(1)}_{r} = \frac{P_{r}(1!c_{1}^{(1)},\ldots, r!c_{r}^{(1)})}{r!},
\end{equation}
so the $b^{(1)}_{r}$ can be deduced from the $c^{(1)}_{r}.$ Let $y_{r}(n)$ denote the coefficient of $q^{n}$ in $(DG_{2}(\tau))^{r}.$ Writing
\begin{displaymath}
a_{i} = (-1)^{i-1}(i-1)!(D_{i}\partial + E_{i}k + F_{i}s + G_{i}x) = (-1)^{i-1}(i-1)!(F_{i}-G_{i})s,
\end{displaymath}
we get the equality
\begin{equation}
\sum_{r=1}^{\infty} \frac{(-1)^{r-1}(F_{r}-G_{r})}{r} \sum_{n=1}^{\infty} y_{r}(n)q^{n} = \sum_{n=1}^{\infty} c^{(1)}_{n}q^{n},
\end{equation}
hence
\begin{equation}
c^{(1)}_{n} = \sum_{r= 1}^{\infty} y_{r}(n)\frac{(-1)^{r-1}(F_r-G_r)}{r}.
\end{equation}
Thus, $B_{1}(q)$ can be deduced from the $a_{i},$ and a similar argument holds for $B_{2}(q).$ This motivates a further study of the $a_i;$ in particular, we include what seems to be an interesting observation. Recall that for each $n \geq 1,$ we can write
\begin{equation}
a_{n}(\partial, k, s, x) = (-1)^{n-1}(n-1)!\left(D_{n}\partial + E_{n}k + F_{n}s + G_{n}x\right)
\end{equation}
for integers $D_{n}, E_{n}, F_{n}, G_{n}.$

Define sequences $D := \{D_{n+1}/D_{n}\}_{n \geq 1},$ $E := \{E_{n+1}/E_{n}\}_{n \geq 1},$ etc. The first terms of these sequences are collected in Table \ref{table:quotients}. In light of these values, we propose the following conjecture:

\begin{conjecture}\label{conj:division}
The four sequences $D,E,F$ and $G$ defined above are convergent.
\end{conjecture}

Provided convergence can be proved, it would be interesting to at least know whether all four sequences converge towards the same number (which, it would seem, is approximately equal to 20, at least for $D,E$ and $F$).

\begin{table}
\centering
\begin{tabular}{|l||l|l|l|l|}
\hline
$n$ & $D_{n+1}/D_{n}$ & $E_{n+1}/E_{n}$ & $F_{n+1}/F_{n}$ & $G_{n+1}/G_{n}$ \\
\hline\hline
1 & 	14	& 19,5		& ---    & 7 \\
\hline			
2 & 	16,43	& 20,21	& 31,33 & 9,86 \\
\hline		
3 & 	17,48	& 20,23	& 25,57 & 9,39 \\
\hline			
4 & 	18,05	& 20,19	& 23,61 & 5,43 \\
\hline			
5 & 	18,42   & 20,14	& 22,62 & 18,77 \\
\hline			
6 & 	18,67	& 20,11	& 22,04 & 51,89 \\
\hline			
7 &	18,86	& 20,09	& 21,67 & 29,93 \\
\hline			
8 & 	19,01	& 20,08	& 21,40 & 25,54 \\
\hline  			
9 &	19,12	& 20,07	& 21,21 & 23,71 \\
\hline			
10 &	19,21	& 20,06	& 21,06 & 22,73 \\
\hline			
11 &	19,29	& 20,06	& 20,95 & 22,13 \\
\hline			
12 & 	19,36	& 20,06	& 20,85 & 21,73 \\
\hline			
13 &	19,41 	& 20,06	& 20,78 & 21,45 \\
\hline			
14 & 	19,46 	& 20,06	& 20,72 & 21,24 \\
\hline
\end{tabular}
\caption{Sequences $D_{n+1}/D_{n}, E_{n+1}/E_{n}, F_{n+1}/F_{n}, G_{n+1}/G_{n}.$}
\label{table:quotients}
\end{table}

\vspace{5mm}

We now relate the polynomials $a_i$ to Kazarian's Thom polynomials, studied in \cite{Kaz},  Kazarian studies, in a topological setting, topological Thom polynomials for multisingularities of a map of manifolds $f: M \rightarrow N.$ In particular, he considers the situation where $f$ is the map from $X,$ the critical locus inside $F = S \times |\mathscr{L}|,$ to $Y = |\mathscr{L}|.$ For each type of multisingularity $\underline{\alpha}$ of small codimension, he introduces and computes an associated integral, linear polynomial in the four Chern numbers of $(S,\mathscr{L}),$ which he denotes by $S_{\underline{\alpha}}.$

\begin{theorem}\label{thm:gen_shape}
\emph{(\cite{Kaz}, Theorem 10.1.)}
For each type $\underline{\alpha} = (\alpha_{1},\ldots,\alpha_{r})$ of multisingularity, the number of curves on $S$ lying in a sufficiently generic linear system $|\mathscr{L}|$ and passing through $N - \textnormal{codim }\underline{\alpha}$ points in general position (where $N$ is the dimension of $|\mathscr{L}|$) is given by
\begin{equation}\label{eqn:gen_shape}
N_{\underline{\alpha}}(S,\mathscr{L}) = \frac{1}{\# \textnormal{Aut}(\underline{\alpha})} \sum_{J_{1} \sqcup \ldots \sqcup J_{l} = [r]} \prod_{i=1}^{l} S_{\underline{\alpha}_{J_i}}.
\end{equation}
\end{theorem}

In particular, we recover the expression of node polynomials as Bell polynomials. Indeed, Theorem \ref{thm:gen_shape} implies that
 \begin{eqnarray*}
N_{r}(S,\mathscr{L}) & = & \frac{1}{\# \textnormal{Aut}(A_{1}^{r})} \sum_{J_{1} \sqcup \ldots \sqcup J_{l} = [r]} \prod_{i=1}^{l} S_{A_{1}^{|J_i|}} \\
& = & \frac{1}{r!}\sum_{l=1}^{r} \sum_{j_{1} + \ldots + j_{r-l+1} = r} e_{j_{1},\ldots, j_{r-l+1}} \prod_{i=1}^{r-l+1} S_{A_{1}^{j_{i}}},
\end{eqnarray*}
with $e_{j_{1},\ldots, j_{r-l+1}}$ the number of ways to partition a set of $r$ elements into $l$ blocks of which $j_1$ have 1 element, $j_2$ have 2 elements, etc. But this is exactly the definition of Bell polynomials, so we get
\begin{equation}
N_{r}(S,\mathscr{L}) = \frac{1}{r!}P_{r}(S_{A_{1}}, \ldots, S_{A_{1}^{r}}).
\end{equation}

Thus, Kazarian's polynomial $S_{A_{1}^{i}}$ corresponds to the polynomial $a_{i}$ of Kleiman--Piene, introduced in \cite{KP1}. We defined $a_{i}$ as the degree of the pushdown to $Y$ of the contribution to $X_{1} \cdot \ldots \cdot X_{i}$ coming from all distinguished varieties with support in the small diagonal $\Delta^{(i)}_{X}.$ We will now summarize the geometric interpretation of the polynomials $a_{i}.$

In \cite{LiTz}, Li and Tzeng prove algebraically the existence of enumerative polynomials for curves with singularity type $\underline{\alpha}.$ However, the form promised by Theorem \ref{thm:gen_shape} is not, a priori, clear from the point of view of algebraic geometry and intersection theory. For the sake of the discussion, we'll assume that such a form is valid in the algebro-geometric setting. Recall that $f: X \rightarrow Y$ is the composition of the embedding $\iota: X \hookrightarrow F$ and the projection $F = S \times Y \rightarrow Y.$ If $\mathscr{L}$ is sufficiently ample on $S,$ we conjecture that
\begin{equation}\label{eqn:expr_nodal}
N_{r}(S,\mathscr{L}) = \frac{1}{r!} \int_{Y} f_{\ast} m_{r} - \sum_{\underline{\alpha} \in \Gamma_{r}^{\circ}} N_{\underline{\alpha}}(S,\mathscr{L}),
\end{equation}
where $m_{r}$ is the $r$-point cycle class of $f$ and $\Gamma_{r}^{\circ}$ is the set of all multisingularity types of codimension $r,$ $A_{1}^{r}$ excepted. Indeed, $m_{r}$ enumerates the $r$-fold points of $f,$ which includes curves with other codimension $r$ multisingularities than $r$ nodes. Using Remark \ref{rem:multiple_bell}, we can rewrite this conjectural equality as
\begin{displaymath}
N_{r}(S,\mathscr{L}) = \frac{1}{r!}P_{r}\left(Q_{1} + C_{1}, \ldots, (-1)^{r-1}(r-1)!(Q_{r} + C_{r})\right) - 
\sum_{\substack{\underline{\alpha} \in \Gamma_{r}^{\circ} \\ J_{1} \sqcup \ldots \sqcup J_{l} = [l(\alpha)]}} \frac{\prod_{i=1}^{l} S_{\underline{\alpha}_{J_i}}}{\# \textnormal{Aut}(\underline{\alpha})},
\end{displaymath}
at least for $r \leq 4.$ Since this should be equal to $\frac{1}{r!}P_{r}(a_{1},\ldots,a_{r}),$ it follows, by comparing the linear terms on each side, that
\begin{equation}
\boxed{a_{i} = (-1)^{i-1}(i-1)!(Q_{i} + C_{i}) - \sum_{\underline{\alpha} \in \Gamma_{i}^{\circ}}\frac{i!}{\# \textnormal{Aut}(\underline{\alpha})}S_{\underline{\alpha}}.}
\end{equation}
We have used the convention $C_1 = C_2 = 0.$ For completeness, we include Kazarian's polynomials $S_{\underline{\alpha}}$ for all $\underline{\alpha}$ with codimension $\leq 4$ in Table \ref{table:kazarian}.

\begin{table}
\centering
\begin{tabular}{|l||l|l|}
\hline
& $\underline{\alpha}$ & $S_{\underline{\alpha}}$ \\
\hline
$\textnormal{cod}(\underline{\alpha}) = 1$ & $A_{1}$ & $3 \partial + 2k + x$ \\
\hline
$\textnormal{cod}(\underline{\alpha}) = 2$ & $A_{2}$ & $12 \partial + 12k + 2s + 2x$ \\
& $A_{1}^2$ & $ -42 \partial - 39 k - 6 s - 7 x$ \\
\hline
$\textnormal{cod}(\underline{\alpha}) = 3$ & $A_3$ & $50 \partial + 64 k + 17 s + 5 x$ \\
& $A_{1}A_{2}$ & $-240 \partial - 288 k - 72 s - 24 x$ \\
& $A_{1}^{3}$ & $1380 \partial + 1576 k + 376 s + 138 x$ \\
\hline
$\textnormal{cod}(\underline{\alpha}) = 4$ & $A_4$ & $180 \partial + 280 k + 100 s$ \\
& $D_4$ & $15\partial + 20 k + 5 s + 5 x$ \\
& $A_{1}A_{3}$ & $-1260\partial - 1820 k - 596 s - 60 x$ \\
& $A_{2}^{2}$ & $-1260 \partial - 1800 k - 588 s - 48 x$ \\
& $A_{1}^{2}A_{2}$ & $9000 \partial + 12360 k + 3864 s + 456 x$\\
& $A_{1}^{4}$ & $-72360 \partial - 95670 k - 28842 s - 3888 x$\\
\hline
\end{tabular}
\caption{The polynomials $S_{\alpha}$ for $\textnormal{codim}(\alpha) \leq 4.$}
\label{table:kazarian}
\end{table}

We therefore see that $a_{i}$ accumulates diverse ``corrections.'' The term $$(-1)^{i-1}(i-1)!(Q_{i}+C_{i})$$ handles the contribution of the small diagonal to the intersection product $X_{1} \cdot \ldots \cdot X_{i},$ while the remaining term handles curves with higher singularities appearing in the correct codimension $i.$

\begin{example}\label{ex:enumerations}
For instance, we have
\begin{eqnarray*}
a_{2} & = & -(Q_{2}+2S_{A_{2}}) \\
& = & -42\partial - 39k - 6s - 7x, \\
a_{3} & = & 2(Q_{3}+C_{3}) - 6(S_{A_{1}A_{2}}+S_{A_{3}}) \\
& = & 1380\partial + 1576k + 376s + 138x; \\
a_{4} & = & -6(Q_4+C_4) - 24(S_{A_{1}A_{3}}+1/2S_{A_{1}^{2}A_{2}}+1/2S_{A_{2}^{2}}+S_{A_{4}}+S_{D_{4}}) \\
& = & -72360\partial - 95670k - 28842s - 3888x,
\end{eqnarray*}
where we have used the numerical expressions for the $S_{\underline{\alpha}}$ provided in \cite{Kaz} and reproduced in Table \ref{table:kazarian}. Thus, the conjectural equality presented in Eq. (\ref{eqn:expr_nodal}) is true up to at least $r=4.$

Terms such as $S_{A_{1}A_{2}}$ also have concrete interpretations. Assume one wants to compute the number of curves in $|\mathscr{L}|$ having one node and one cusp, and passing through $N-3$ points in general position on $S.$ The configuration space of choice for this computation is $F^{\times 2}.$ Let $C \subset X$ denote the locus of curves with a marked singularity which is a cusp \textit{or worse}. We are, a priori, interested in the intersection product $p_{1}^{\ast}[C] \cdot p_{2}^{\ast}[X],$ but there is an excess contribution from the diagonal $\Delta_{C} \cong C,$ as well as an embedded component related to tacnodal curves. In the case of $(\mathbb{P}^{2}, \mathscr{O}_{\mathbb{P}^{2}}(d)),$ we can compute explicitly the excess contribution, using results from \cite{Alu2}. Indeed, according to \cite[Lemma 1.4]{Alu2}, we have
\begin{equation}
c(N_{C}X) = 1 + 2(d-3)l + 2H,
\end{equation}
where $l$ denotes the class of a hyperplane in $\mathbb{P}^{2}$ and $H$ denotes the class of a hyperplane in the $\mathbb{P}^{N}$ of curves of degree $d.$ So we get
\begin{eqnarray*}
c(N_{X}F) & = & (1+(d-1)l + H)^{3} \\
c(N_{C}F) & = & (1 + 2(d-3)l + 2H)(1+(d-1)l + H)^{3}.
\end{eqnarray*}
Hence, since $c(N_{F}F^{\times 2})^{-1} = 1-3l+6l^2,$ the equivalence $E_{A_{1}A_{2}}$ of $\Delta_{C}$ for $p_{1}^{\ast}[C] \cdot p_{2}^{\ast}[X]$ is the coefficient of $l^{2}H^{3}$ in
\begin{equation}
(1+(d-1)l + H)^{3} \cdot (1-3l+6l^2) \cap (2(d-3)l+2H)((d-1)l + H)^{3}.
\end{equation}
A quick computation, using for instance \verb+Maple+, shows that this is equal to $60d^{2}-192d+144.$ Since $S_{A_{1}A_{2}} = -240d^{2}+864d-720$ and $S_{A_{3}} = 50d^{2}-192d+168,$ we see that
\begin{equation}
S_{A_{1}A_{2}} = -3(1/2E_{A_{1}A_{2}}+S_{A_{3}}),
\end{equation}
and the number of curves with a cusp and a node is given by
\begin{equation}
N_{A_{1}A_{2}} = S_{A_{1}}S_{A_{2}}+S_{A_{1}A_{2}}.
\end{equation} \hfill $\blacksquare$
\end{example}

\end{document}